\documentclass[11pt]{article}

\usepackage[T1]{fontenc}
\usepackage{verbatim}
\usepackage{amsmath,amsthm}
\usepackage{xspace}
\usepackage{pifont}
\usepackage{amssymb}
\usepackage{dsfont}
\usepackage{makeidx}
\usepackage{mathrsfs}
\usepackage{xcolor, colortbl}
\usepackage{subcaption}
\usepackage[numbers]{natbib}
\usepackage{enumitem}
\usepackage{afterpage}
\usepackage{mathtools}
\usepackage{graphics}
\usepackage[export]{adjustbox}
\usepackage{tcolorbox}
\usepackage{authblk}
\usepackage[normalem]{ulem}
\usepackage[colorlinks,citecolor=blue,urlcolor=blue]{hyperref}
\mathtoolsset{showonlyrefs}

\def\eps{\varepsilon}

\newcommand{\R}{{\mathbb R}}

\newcommand{\N}{{\mathbb N}}

\newcommand{\bd}{\begin{displaymath}}
\newcommand{\ed}{\end{displaymath}}
\newcommand{\be}{\begin{equation}}
\newcommand{\ee}{\end{equation}}
\newcommand{\bq}{\begin{eqnarray}}
\newcommand{\eq}{\end{eqnarray}}
\newcommand{\bn}{\begin{eqnarray*}}
\newcommand{\en}{\end{eqnarray*}}

\newtheorem{theorem}{Theorem}[section]
\newtheorem{lemma}[theorem]{Lemma}
\newtheorem{proposition}[theorem]{Proposition}

\newtheorem{corollary}[theorem]{Corollary}

\newtheorem{remark}[theorem]{Remark}

\newtheorem{example}[theorem]{Example}

\newtheorem{definition}[theorem]{Definition}
\newtheorem{assumption}[theorem]{Assumption}
\numberwithin{equation}{section}

\begin{document}

\title{Potential Games on Unimodular Random Graphs}
\author[]{Eyal Neuman}
\author[]{Sturmius Tuschmann\thanks{ST is supported by the EPSRC Centre for Doctoral Training in Mathematics of Random Systems: Analysis, Modelling and Simulation (EP/S023925/1).}}
\affil[]{Department of Mathematics, Imperial College London}
\renewcommand\Authands{ and }
\date{\today}
\maketitle

\begin{abstract}
We study potential games on unimodular random graphs of bounded degree, where players interact through the underlying network. Using the unimodular measure, we define a well-posed global potential that captures both finite- and infinite-player games. A key observation is that the mass-transport principle identifies the first variation of this potential with the first-order condition of a representative (root) player. Under suitable convexity assumptions, we prove that minimizers of the potential coincide with quenched Nash equilibria, and conversely. We also establish the thermodynamic limit of the potential along weakly convergent sequences of unimodular measures. Finally, we present examples with semi-explicit equilibrium descriptions. In linear-quadratic games on unimodular graphs, equilibria are expressed in terms of the Green kernel of the simple random walk operator, while in convex settings, equilibria are characterized by solutions to nonlinear Poisson equations.
\end{abstract}

\begin{description}
\item[Mathematics Subject Classification (2020):] 	91A07, 91A43, 93E20 
\item[Keywords:] network games, sparse graphs, potential games, unimodular measures, local weak convergence, representative player, random walks on graphs
\end{description}

\section{Introduction}

Network games are games in which interactions between players are heterogeneous, often modeled through a graph. Each player’s action affects not only their own payoff, but also the local environment and the incentives of neighboring players. Such models arise naturally in a broad range of real-world applications, including autonomous driving \cite{hang2020human}, dynamic production management \cite{leng2005game}, and real-time bidding \cite{sayedi2018real}. From a mathematical perspective, their analysis is challenging because it must simultaneously account for heterogeneous interactions, large populations, and nonlinear feedback effects between individual actions and the environment.

In recent years, there has been substantial progress in the analysis of static and dynamic games on dense networks using the graphon framework (see, e.g., \cite{aurell2022stochastic,caines2019graphon,caines2021graphon,carmona2022stochastic, neuman2024stochastic,neuman2025stochastic, parise2021analysis, parise2023graphon, tangpi2024optimal}). A central advantage of the mean-field game (MFG) paradigm is that it simplifies the analysis of large games by reducing an $n$-player interaction problem to that of a single representative player via a fixed-point condition. Following this idea, \citet{lacker2023label} showed that a graphon game can be recast as a standard MFG on an enlarged state space, where each player is described not only by its state variable but also by a random label encoding its position in the graphon. As a consequence, equilibria can be characterized by a single forward–backward system rather than by a continuum of coupled equations, thereby recovering a level of tractability akin to that of classical MFGs.

At the same time, most real-world networks are sparse, as players typically maintain only a limited number of interactions due to informational, economic, or technological constraints. This sparsity plays a fundamental role in shaping strategic behavior, as it restricts the range of interactions and gives rise to strong locality effects. Consequently, dense network models often fail to capture key features of large-scale systems encountered in economics, engineering, and the social sciences. In contrast, sparse network models, in which node degrees remain locally bounded, provide a more realistic description of such environments and are therefore central to the study of large-population games.

Despite their importance, the theoretical understanding of games on sparse networks remains limited. While dense graph limits are well described by graphon theory \cite{lovasz2012large,lovasz2006limits}, the sparse regime requires fundamentally different tools. A natural framework for sparse graph limits is local weak convergence, introduced by \citet{benjamini2001recurrence} and independently by \citet{aldous2004objective}. This notion captures the asymptotic distribution of neighborhoods around a typical node and leads to probability measures on the space of countable, connected, locally finite, rooted graphs. In this setting, unimodular probability measures emerge as canonical limit objects, playing a role analogous to that of graphons in the dense case \cite{Aldous2007, benjamini2015unimodular}. Their defining property is the mass–transport principle (see \eqref{eq:MTP}). The framework of local weak convergence has already proved fruitful in the study of interacting particle systems on sparse graphs (see, e.g., \citep{ganguly2024hydrodynamic, guo2025particle, LackerRamananWu2021, LackerRamananWu2023b, lacker2023local, lacker2025marginal, oliveira2019interacting, OliveiraReisStolerman2020, ramananICM}). More recently, the authors of the present paper used this approach to analyze the convergence of Nash equilibria in games on sparse graphs \cite{NT26}. Motivated by the MFG perspective, this raises the following fundamental question about network games in the infinite-graph limit:
\begin{center}
\emph{Is there an analogue of the representative player in network games?}
\end{center}

We answer this question within the class of potential games, a particularly tractable and extensively studied class of games. Originally introduced in the static setting by \citet{monderer1996potential}, potential games have received renewed attention in recent years, partly through developments in multi-agent reinforcement learning (see, e.g., \citep{Busoniu2008,DiHuWangZhang2025,Littman1994,Marden2012,PlankZhang2026,ZhangBasar2021}) and partly through recent analytical advances in dynamic potential games \cite{GuoZhang2025,GuoLiZhang2024,GuoLiZhang2025Jumps}. A game is called a potential game if there exists a global functional, called a potential, such that every unilateral deviation by a player changes this functional by exactly the same amount as it changes that player’s payoff. This variational structure makes it possible to reformulate the search for Nash equilibria as an optimization problem for the potential, thereby opening the door to powerful analytical and computational methods.

While potential games have been extensively studied in finite-player and mean-field settings, their extension to graph-based and infinite-player models raises substantial challenges. In network settings, payoffs depend on local neighborhoods, and there is generally no canonical global state variable, which complicates both the construction and the identification of potential structures. Moreover, even when local interactions admit a potential representation, it is not clear that these local structures can be aggregated into a well-defined global potential on the graph, since this would require summing contributions from infinitely many players. On the other hand, the study of potential games on infinite random graphs is particularly relevant, as it is often unrealistic to assume that each player has detailed knowledge of the entire network. Collecting exact network data may be prohibitively costly and often impossible because of proprietary or privacy constraints \cite{parise2023graphon}. A more realistic approach is therefore to model the network through a probability distribution over graph realizations, such as a unimodular measure. This leads to the second question addressed in this paper:
\begin{center}
\emph{Can one define a potential game on an infinite random graph?}
\end{center}

The goal of this paper is to develop a general framework for potential games on unimodular random graphs, thereby bridging the theories of sparse graph limits and potential games. We consider games in which players are indexed by the vertices of a bounded-degree graph, possibly infinite, and interact through the underlying network. The players' actions take values in a general Hilbert space, a formulation flexible enough to cover static games as well as dynamic games in either discrete or continuous time. See Section~\ref{sec:model} for details.

In Section~\ref{sec-results}, we introduce a potential functional $\Phi$ for finite and infinite-player games, defined through the unimodular measure. Proposition~\ref{prop:Phi-finite} establishes that this potential is finite for all admissible actions, despite the possible presence of infinitely many players. Proposition~\ref{prop:Phi-Gateaux} then identifies its first variation with the first-order condition of a representative player via the mass–transport principle. In particular, Theorem~\ref{thm:main} shows that, under convexity with respect to each player's own action, minimizers of the potential \(\Phi\) form quenched Nash equilibria. Conversely, if $\Phi$ is itself convex, then every quenched Nash equilibrium is also a minimizer of $\Phi$. The mass-transport principle is again the key mechanism underlying this equivalence. In Theorem~\ref{thm:thermo-limit}, we then show that whenever a sequence of finite graphs converges in the local weak sense, the associated finite-player potentials converge to the potential of the limiting unimodular random graph, further justifying our definition of the infinite-player potential.

In Section~\ref{sec:examples}, we present examples in which the game admits semi-explicit solutions. In particular, Proposition~\ref{prop:lq-example} and Corollary~\ref{cor:green} show that, for linear–quadratic potential games on unimodular random graphs whose random walk operator has spectral radius strictly less than $1$, equilibria can be expressed in terms of the Green kernel of the graph. In Proposition~\ref{prop:W-example}, we further show that, in the nonlinear convex case, the equilibrium is characterized by a solution to the Poisson equation associated with a nonlinear Laplacian induced by the random graph realization, together with the convex component of the potential. Throughout, the mass-transport principle and its interplay with the potential play a central role in deriving these results. Sections~\ref{sec:proofs-Phi}--\ref{sec:proofs-examples} are devoted to the proofs of our results.

\section{Model Setup} \label{sec:model}

In this section, we set up a general framework for potential games on unimodular random graphs. We begin by introducing unimodularity, then present the game-theoretic model, and finally discuss several examples of graphs that illustrate the representative-player formulation.

\subsection{Unimodular random graphs}

We consider graphs $G=(V,E)$ that consist of a finite or countably infinite vertex set $V$ and an edge set $E$. All graphs are assumed to be simple (that is, they are undirected and have no loops or multiple edges) and locally finite (that is, the degree $\deg_G(v)$ of each vertex $v\in V$ is finite). For two vertices $u,v\in V$, denote by $d_G(u,v)$ the graph distance between $u$ and $v$, that is, the length of the shortest path in $G$ from $u$ to $v$. If $\{u,v\}\in E$, we write $u\sim v$ and call $u$ and $v$ neighbors. With a slight abuse of notation, we often write $v\in G$ instead of $v\in V$, and set $|G|:=|V|$ for the cardinality of the vertex set.

A rooted graph $(G,o)=(V,E,o)$ consists of a (simple, locally finite) graph $G=(V,E)$ with a distinguished vertex $o\in V$, called the root. Given a rooted graph $(G,o)$ and $r\in\mathbb N_0$, we write $B_r(G,o)$ for the induced rooted subgraph consisting of all vertices $v\in V$ whose graph distance $d_G(o,v)$ from the root $o$ is at most $r$. Two rooted graphs $(G,o)=(V,E,o)$ and $(G',o')=(V',E',o')$ are said to be isomorphic if there exists a bijection $\varphi:V\to V'$ with $\varphi(o)=o'$ such that
$$
\{u,v\}\in E \,\Longleftrightarrow\, \{\varphi(u),\varphi(v)\}\in E',\quad\text{for all }u,v\in V.
$$
In this case we write $(G,o)\cong (G',o')$ and say that $\varphi$ is an isomorphism from $(G,o)$ to $(G',o')$. We denote the isomorphism class of a rooted graph $(G,o)$ by $[G,o]$.
\begin{definition}\label{def:local-convergence}
Let $\mathcal G_\ast$ denote the set of isomorphism classes of countable, connected, rooted graphs. A sequence $\{[G_n,o_n]\}_{n\in\mathbb N}\subset\mathcal G_\ast$ is said to converge locally to $[G,o]\in\mathcal G_\ast$ if for every $r\in\mathbb N_0$ there exists $n_r\in\mathbb N$ such that $B_r(G_n,o_n)\cong B_r(G,o)$ for all $n\ge n_r$. 
\end{definition}
For $[G,o],[G',o']\in \mathcal G_\ast$, define the corresponding metric by
\be\label{eq:d_ast}
d_\ast\big([G,o],[G',o']\big):=\inf_{r\in\mathbb N_0}\big\{2^{-r}:B_r(G,o)\cong B_r(G',o')\big\}.
\ee
\begin{remark}
The metric in \eqref{eq:d_ast} turns $(\mathcal G_\ast, d_\ast)$ into a complete and separable metric space (see \cite{aldous2004objective}, Section~2.2). If one restricts $\mathcal G_\ast$ to graphs of uniformly bounded degree, $(\mathcal G_\ast, d_\ast)$ is compact (see \cite{lovasz2012large}, Chapter~18.3.1).
\end{remark}

Similarly to Definition~\ref{def:local-convergence}, let \(\mathcal G_{\ast\ast}\) denote the space of isomorphism classes $[G,o_1,o_2]$ of countable, connected, doubly rooted graphs \((G,o_1,o_2)\). Let $d_{\ast\ast}$ be the corresponding metric defined analogous to \eqref{eq:d_ast} (two doubly rooted graphs \((G,o_1,o_2)\) and \((G',o_1',o_2')\) are isomorphic if there exists an isomorphism \(\varphi:(G,o_1)\to(G',o_1')\) with $\varphi(o_2)=o_2'$). Assume that $\mathcal{G}_\ast$ and $\mathcal{G}_{\ast\ast}$ are equipped with their Borel \(\sigma\)-fields induced by $d_\ast$ and $d_{\ast\ast}$.

\begin{definition}
Let \(\mu\) be a probability measure on \(\mathcal G_\ast\). We say that \(\mu\) is unimodular if for every nonnegative measurable function $T:\mathcal G_{\ast\ast}\to [0,\infty]$ the following mass-transport principle holds:
\begin{equation} \label{eq:MTP}
\int_{\mathcal G_\ast} \sum_{v\in V(G)} T([G,o,v])\mu(d[G,o])=
\int_{\mathcal G_\ast} \sum_{v\in V(G)} T([G,v,o])\mu(d[G,o]).
\end{equation}
For a unimodular measure \(\mu\), a random rooted graph \([G,o]\) distributed according to $\mu$ is called a unimodular random graph.
\end{definition}
\begin{remark}\label{rem:local-weak-convergence}
Unimodular measures are central in the theory of sparse graphs, as they arise as limits of locally weakly convergent graph sequences (see \cite{benjamini2001recurrence}, Section~3.2). More precisely, a sequence of finite graphs \(\{G_n\}_{n\in\N}\) is said to converge locally weakly to a probability measure \(\mu\) on \(\mathcal G_\ast\) if the probability measures \(\{\mu_{G_n}\}_{n\in\N}\) on \(\mathcal G_\ast\), obtained by choosing a root uniformly at random and considering the associated rooted connected component, converge weakly to \(\mu\). We do not further formalize local weak convergence here, and only note that this notion can also be extended to random graph sequences. See \cite{Aldous2007,aldous2004objective,benjamini2001recurrence,bordenave2016lecture,van2024random} for further details.
\end{remark}

In what follows, we impose the following bounded-degree assumption.
\begin{assumption}\label{ass:degree-bound}
Let $\mu$ be a unimodular measure on $\mathcal{G}_\ast$. We assume that there exists a constant \(D\in\mathbb N\) such that
\be\label{eq:degree-bound}
\deg_G(v)\le D, \quad\text{for all } v\in V(G), \text{ for } \mu\text{-a.e. } [G,o].  
\ee
\end{assumption}

\subsection{The potential network game}

Throughout, let \(H\) denote a separable real Hilbert space with inner product \(\langle \cdot,\cdot\rangle_H\) and norm \(\|\cdot\|_H\). The set of admissible actions \(K\subset H\) is assumed to be nonempty, closed, and convex.

A measurable map \(x:\mathcal G_\ast\to H\) will be interpreted as an action rule.
Given such \(x\) and a rooted graph \((G,o)\), define 
\be\label{eq:xGv}
x^G_v:=x([G,v]), \quad v\in V(G).
\ee
Thus, a single action rule on the space $\mathcal{G}_\ast$ induces an action profile \(x^G=(x^G_v)_{v\in V(G)}\) on each countable, connected (unrooted) graph \(G\). We define the Hilbert space
\be \label{eq:H} 
\mathcal H:=L^2(\mu,H)=
\bigg\{x:\mathcal G_\ast\to H \text{ measurable} \,\bigg|\, 
\|x\|_{\mathcal H}^2:=\int_{\mathcal G_\ast}\|x([G,o])\|_H^2\mu(d[G,o])<\infty\bigg\},
\ee
endowed with the inner product
\[
\langle x,y\rangle_{\mathcal H}
:=
\int_{\mathcal G_\ast} \langle x([G,o]),y([G,o])\rangle_H\mu(d[G,o]),\quad x,y\in\mathcal{H}.
\]
Let
\be\label{eq:K}
\mathcal K:=\big\{x\in \mathcal H \,\big|\,  x([G,o])\in K \text{ for } \mu\text{-a.e. } [G,o]\big\}.
\ee
Since \(K\subset H\) is closed and convex, \(\mathcal K\subset \mathcal H\) is also closed and convex. Next, let
\be\label{eq:fh}
f:\mathcal G_\ast\times H\to \mathbb R
\quad\text{and}\quad
h:\mathcal G_{\ast\ast}\times H\times H\to \mathbb R
\ee
be measurable functions. We interpret \(f\) as the private cost and \(h\) as the interaction cost. For a fixed countable, connected graph \(G\) and an action profile \(\smash{x^G=(x^G_u)_{u\in V(G)}\in K^{V(G)}}\), define the cost functional $J_v^G$ of player \(v\in V(G)\) by
\begin{equation} \label{eq:J}
J_v^G(x^G)
:=
f([G,v],x^G_v)+\sum_{u\sim v} h([G,v,u],x^G_v,x^G_u).
\end{equation}
The corresponding game on the fixed countable, connected graph \(G\) will be referred to as the quenched game.

\begin{remark}\label{rem:examples-H}
The abstract Hilbert space \(H\) can be chosen in different ways, depending on whether one wants to model static or dynamic games. In particular, our framework covers the following examples.
\begin{enumerate}[label=(\roman*)]
\item \emph{Static games.}  
One may take \(H=\mathbb R^m\) for some \(m\in\mathbb N\), equipped with its usual Euclidean inner product.
\item \emph{Dynamic games.}  
Let \(T>0\) and fix a filtered probability space \((\Omega,\mathcal F,\mathbb F=(\mathcal F_t)_{0\le t\le T},\mathbb P)\). 
\begin{itemize}
\item \emph{Discrete time.}  
Let \(\mathbb T:=\{0=t_0<t_1<\ldots<t_N=T\}\) be a finite time grid. One may choose
\[
H=\bigg\{ a:\Omega\times\mathbb T\to\mathbb R\,\bigg|\,
a(t_j)\text{ is }\mathcal F_{t_j}\text{-measurable for every $j$, } 
\mathbb E\bigg[\sum_{j=0}^N  a(t_j)^2\bigg]<\infty\bigg\},
\]
equipped with the inner product
$\langle  a, b\rangle_H=\mathbb E[\sum_{j=0}^N  a(t_j) b(t_j)]$ for $a,b\in H$.
\item \emph{Continuous time.}  One may set 
\[
H=\bigg\{ a:\Omega\times[0,T]\to\mathbb R\,\bigg|\,
a \text{ is }\mathbb F\text{-prog.~measurable, }
\mathbb E\bigg[\int_0^T a(t)^2dt\bigg]<\infty\bigg\},
\]
equipped with the inner product
$\langle a,b\rangle_H=\mathbb E[\int_0^T  a(t) b(t)dt]$ for $a,b\in H$.
\end{itemize}
\end{enumerate}
\end{remark}
\begin{remark}
The restriction to action profiles induced by action rules as in \eqref{eq:xGv} is a central feature of the present framework. It ensures consistency under rerooting: whenever two players \(u,v\in V(G)\) have indistinguishable rooted positions in the graph
$G$, in the sense that the rooted graphs $(G,u)$ and $(G,v)$ are isomorphic, they must choose the same action, that is, \(x_u^G=x_v^G\). For graphs with nontrivial automorphisms, this can lead to a substantial reduction in complexity, since vertices lying in the same orbit of the automorphism group necessarily carry the same action under any admissible rule. For highly symmetric graphs, the analysis of the infinite-player game may therefore reduce to a finite-dimensional problem, as illustrated in Examples \ref{ex:cycles}, \ref{ex:boxes}, and \ref{ex:d-regular} below.
\end{remark}

\begin{definition}
Let \(x\in \mathcal K\). We say that \(x\) is a quenched Nash equilibrium if
\begin{equation} \label{eq:Nash}
J_o^G(x^G)\le J_o^G(a,x^G_{-o}),\quad \text{for all } a\in K,\text{ for } \mu\text{-a.e. } [G,o],
\end{equation}
where \(x^G_v=x([G,v])\) for all \(v\in V(G)\), and \((a,x^G_{-o})\) denotes the action profile obtained from \(x^G\) by replacing the root action $x^G_o$ by $a$.
\end{definition}
In other words, for $\mu$-a.e.~rooted graph $[G,o]$, the root player cannot lower their cost by a unilateral deviation, holding fixed the actions prescribed by the rule $x$ at all other vertices.

\begin{remark} Let \(G=(V,E)\) be a finite connected graph with \(|V|=n\in\N\). Then \(G\) induces a unimodular measure \(\mu_G\) on \(\mathcal G_\ast\) by choosing the root uniformly at random:
\be\label{eq:muG}
\mu_G = \frac{1}{n} \sum_{v\in V} \delta_{[G,v]}.
\ee
In this case, an action rule \(x\in \mathcal K\) induces an action profile \(x^G=(x_v^G)_{v\in V}\) on the finite player set \(V\). The quenched Nash equilibrium condition
\eqref{eq:Nash} holds if and only if
\[
J_v^G(x^G)\le J_v^G(a,x^G_{-v}),
\quad \text{for all } a\in K,\ v\in V.
\]
\end{remark}

We now impose the structural assumptions on the cost functionals in \eqref{eq:J} that make this game a potential game.

\begin{assumption}\label{ass:potential}
Assume the following on the cost functionals defined in \eqref{eq:fh} and \eqref{eq:J}.
\begin{enumerate}[label=(\roman*)]
\item For \(\mu\)-a.e.~\([G,o]\), the map \(a\mapsto f([G,o],a)\) is G\^ateaux differentiable on \(H\).

\item For \(\mu\)-a.e.~\([G,o]\) and every neighbor $v\sim o$, the map \(a\mapsto h([G,o,v],a,b)\) is G\^ateaux differentiable on \(H\) for every \(b\in H\). Moreover, there exists a measurable map
\be\label{eq:psi}
\psi:\mathcal G_{\ast\ast}\times H\times H\to \mathbb R,
\ee
such that for \(\mu\)-a.e.~\([G,o]\) and every neighbor $v\sim o$, the map \((a,b)\mapsto \psi([G,o,v],a,b)\) is G\^ateaux differentiable on \(H\times H\), and for all \(a,b\in H\),
\begin{equation} \label{eq:potential-id-1}
\nabla_1 \psi([G,o,v],a,b)=\nabla_1 h([G,o,v],a,b),
\end{equation}
\begin{equation} \label{eq:potential-id-2}
\nabla_2 \psi([G,o,v],a,b)=\nabla_1 h([G,v,o],b,a).
\end{equation}

\item There exist constants \(L_f,L_h\ge 0\) such that for \(\mu\)-a.e.~\([G,o]\), every neighbor \(v\sim o\), and all \(a,a',b,b'\in H\),
\begin{equation} \label{eq:lip-f}
\|\nabla f([G,o],a)-\nabla f([G,o],a')\|_H
\le
L_f\|a-a'\|_H,
\end{equation}
and
\begin{equation} \label{eq:lip-h}
\|\nabla_1 h([G,o,v],a,b)-\nabla_1 h([G,o,v],a',b')\|_H
\le
L_h(\|a-a'\|_H+\|b-b'\|_H).
\end{equation}

\item There exists a reference point \(a_\ast\in H\) such that
\begin{equation} \label{eq:int-base-f}
\int_{\mathcal G_\ast} |f([G,o],a_\ast)|\mu(d[G,o])<\infty,
\quad
\int_{\mathcal G_\ast} \|\nabla f([G,o],a_\ast)\|_H^2\mu(d[G,o])<\infty,
\end{equation}
and
\begin{equation} \label{eq:int-base-psi}
\int_{\mathcal G_\ast}\sum_{v\sim o} |\psi([G,o,v],a_\ast,a_\ast)|\mu(d[G,o])<\infty,
\end{equation}
\begin{equation} \label{eq:int-base-h}
\int_{\mathcal G_\ast}\sum_{v\sim o} \|\nabla_1 h([G,o,v],a_\ast,a_\ast)\|_H^2\mu(d[G,o])<\infty.
\end{equation}
\end{enumerate}
\end{assumption}

\begin{remark}
Assumptions \eqref{eq:potential-id-1}--\eqref{eq:potential-id-2} encode the exact potential property of \citet{monderer1996potential} in differential form at the level of a single edge $\{o,v\}$. Equivalently, in difference form, for all \(a,a',b,b'\in H\),
\be\label{eq:potential-id-1-diff}
\psi([G,o,v],a,b)-\psi([G,o,v],a',b)=h([G,o,v],a,b)-h([G,o,v],a',b),
\ee
and
\be\label{eq:potential-id-2-diff}
\psi([G,o,v],a,b)-\psi([G,o,v],a,b')=h([G,v,o],b,a)-h([G,v,o],b',a).
\ee
That is, a unilateral deviation at $o$ changes the interaction cost $h$ by exactly the corresponding change in the functional $\psi$ when its first argument is varied, see \eqref{eq:potential-id-1-diff}. Reversing the perspective along the edge $\{o,v\}$, a unilateral deviation at $v$ changes the interaction cost $h$ by exactly the corresponding change in the functional $\psi$ when its second argument is varied, see \eqref{eq:potential-id-2-diff}. 
\end{remark}

\subsection{Examples of graphs and associated representative players}

We illustrate the representative-player formulation with several basic examples. Here, the key point is that admissible action profiles in \eqref{eq:xGv} are induced by action rules \(x:\mathcal G_\ast\to H\), so that players with isomorphic rooted neighborhoods always choose the same action.

\begin{example}[Cycle graphs]\label{ex:cycles}
For $n\in\N$, let \(C_n\) be the cycle graph with vertices 
$$
V(C_n)=[n]:=\{1,\dots, n\},
$$ 
and let $\mu_{C_n}$ denote the corresponding unimodular measure on $\mathcal{G}_\ast$ from \eqref{eq:muG}. Since \(C_n\) is transitive, all the rooted graphs \((C_n,k)\) for $k\in [n]$ are isomorphic. Consequently, every action rule induces a constant action profile on \(C_n\). In other words, although the game on $C_n$ has \(n\) players, only a single player type appears in the representative-player formulation.
As \(n\to\infty\), the local weak limit of $\{C_n\}_{n\in\N}$ is the unimodular measure concentrated on the two-way infinite path $P\cong \mathbb Z$ rooted at an arbitrary vertex (see \cite{lovasz2012large}, Chapter~19, Example~19.3). Again, this limit graph is transitive, so the infinite-player game also has only one representative player (see Figure~\ref{fig:cycles}).
\begin{figure}[htb]
    \centering
    \includegraphics[width=0.55\linewidth]{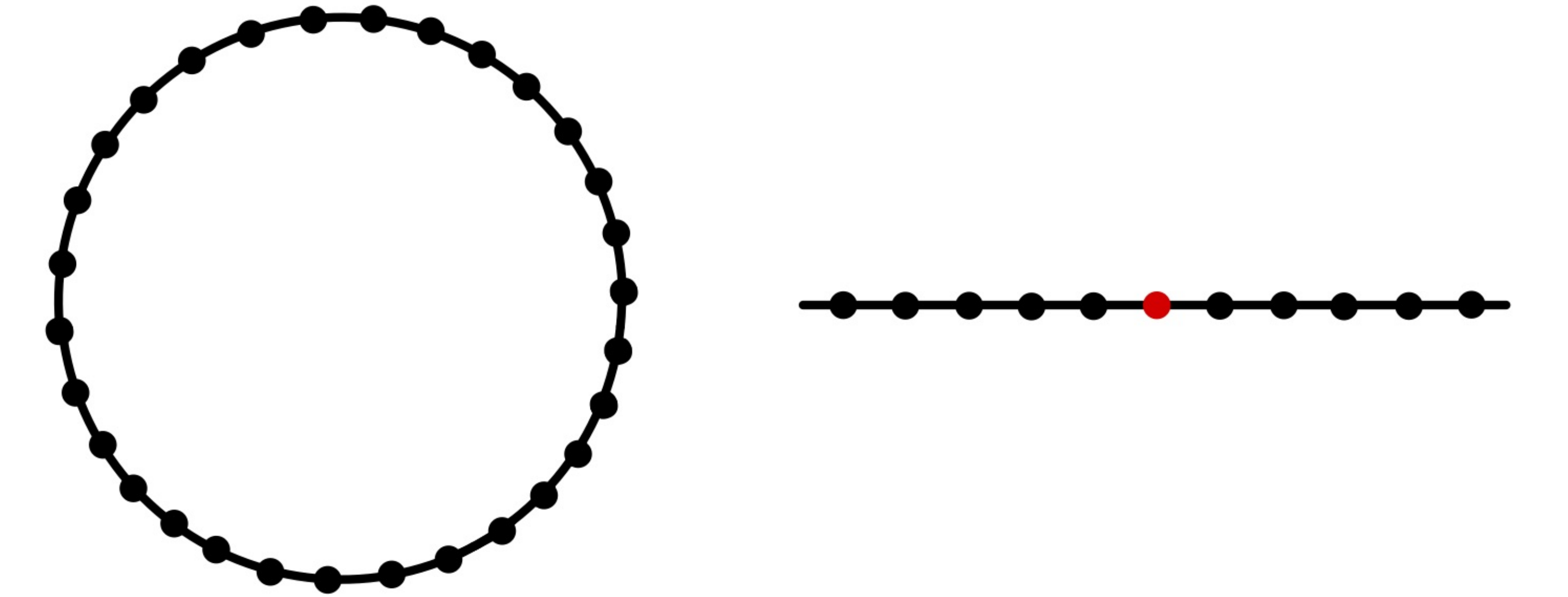}
    \caption{Left: the $n$-cycle $C_n$ on $[n]$ (for $n=30$). Right: its local weak limit, the two-way infinite path $P\cong \mathbb Z$ with a root (red).}
    \label{fig:cycles}
\end{figure}
\end{example}

The next example shows that passing to the infinite-player limit can reduce the dimensionality of the game significantly.

\begin{example}[Boxes in \(\mathbb Z^d\)]\label{ex:boxes}
Fix \(d\in\N\). For each \(n\in\N\), let \(G_n\) be the nearest-neighbor box graph on
\[
V(G_n)=[n]^d:=\{1,\dots,n\}^d,
\]
where two vertices \(u=(u_1,\dots,u_d)\) and \(v=(v_1,\dots,v_d)\) are adjacent if and only if there exists a unique index \(i\in\{1,\dots,d\}\) such that \(|u_i-v_i|=1\) and \(u_j=v_j\) for all \(j\neq i\). Let \(\mu_{G_n}\) be the associated unimodular measure on $\mathcal{G}_\ast$ from \eqref{eq:muG}. For \(v=(v_1,\dots,v_d)\in[n]^d\), define
\[
\delta_i(v):=\min\{v_i-1,n-v_i\},\quad i=1,\dots,d.
\]
Then the isomorphism class $[G_n,v]$ of \((G_n,v)\) is uniquely determined by the multiset
\[
\{\delta_1(v),\dots,\delta_d(v)\}.
\]
Each \(\delta_i(v)\) takes values in \(\{0,\dots,\left\lceil\frac n2\right\rceil -1\}\), so the number of representative players is exactly the number of multisets of size \(d\) drawn from a set of size \(\left\lceil\frac n2\right\rceil\), namely,
\[
\binom{\lceil n/2\rceil+d-1}{d},
\]
by the stars and bars theorem (see \cite{feller1991introduction}, Chapter II.5). Thus, the number of representative players grows like $n^d$ as \(n\to\infty\). However, the local weak limit of $\{G_n\}_{n\in\N}$ is the unimodular measure concentrated on \(\mathbb Z^d\) rooted at an arbitrary vertex (see \cite{van2024random}, Section 2.3.4). Since \(\mathbb Z^d\) is transitive, the limiting infinite-player game has only one representative player (see Figure~\ref{fig:boxes}). 
\begin{figure}[htb]
    \centering
    \includegraphics[width=0.45\linewidth]{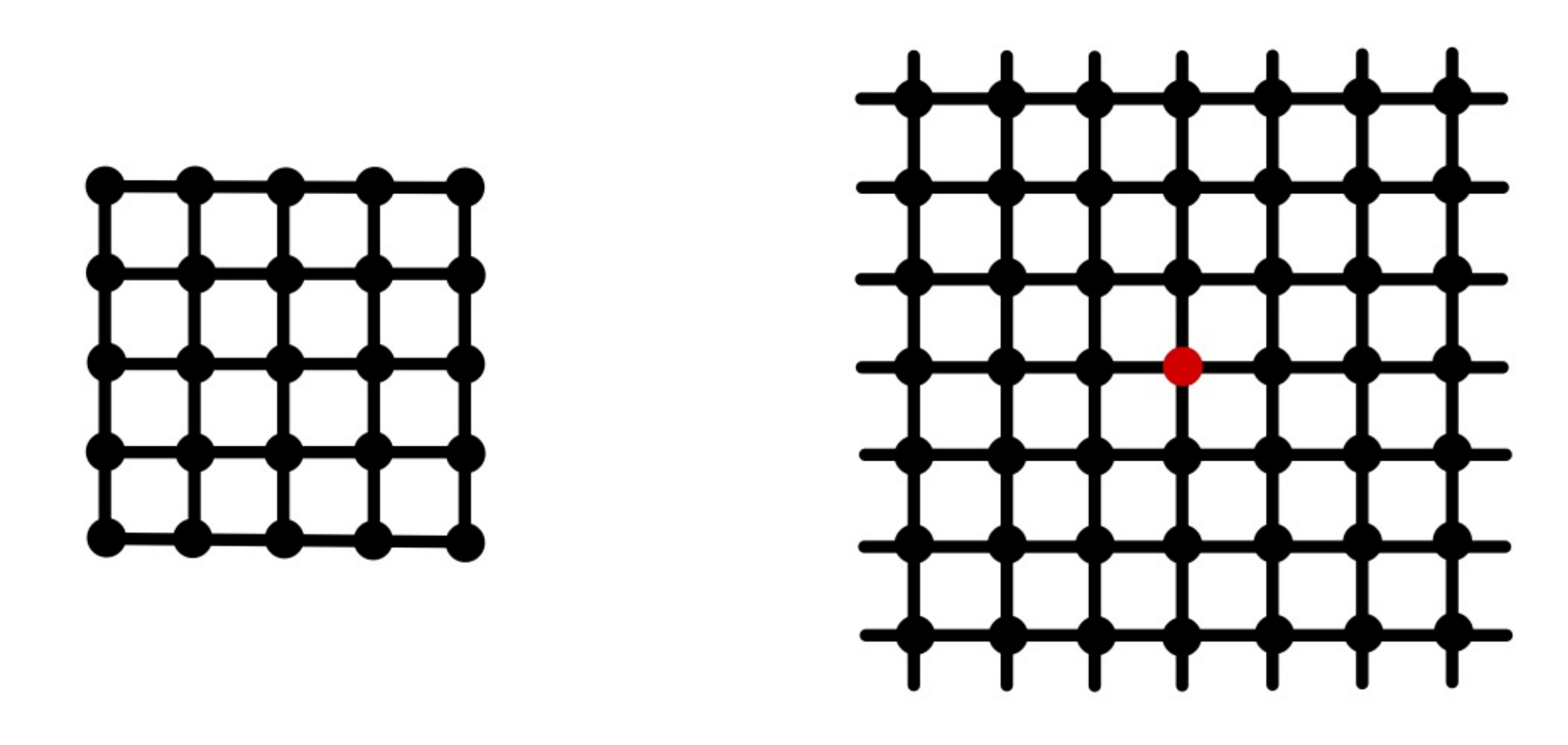}
    \caption{Illustration of the case $d=2$. Left: the nearest-neighbor box graph $G_n$ on $[n]^2$ (for $n=5$). Right: its local weak limit, the two-dimensional lattice $\mathbb Z^2$ with a root (red).}
    \label{fig:boxes}
\end{figure}
\end{example}

\begin{remark}For \(d=1\), Example~\ref{ex:boxes} yields the path graph \(P_n\) with \(\lceil n/2\rceil\) representative players, while for \(d=2\) it yields the grid of size \(n\times n\) with 
$$\binom{\lceil n/2\rceil+1}{2}=\frac{\lceil n/2\rceil(\lceil n/2\rceil+1)}{2}$$
representative players.
\end{remark}

\begin{example}[Random \(D\)-regular graphs]\label{ex:d-regular}
Fix \(D\ge 2\), and for each \(n\in\N\) with \(nD\) even, let \(R_n\) be a random \(D\)-regular graph on \(n\) vertices, that is, a graph chosen uniformly at random from the set of all (simple) graphs on \([n]\) in which every vertex has degree \(D\). For each realization of \(R_n\), let \(\mu_{R_n}\) denote the associated unimodular measure on $\mathcal{G}_\ast$ from \eqref{eq:muG}. In general, for a given realization of $R_n$, the rooted isomorphism classes \((R_n,v)\) for \(v\in V(R_n)\) need not coincide. Thus, in the finite-player game, the representative-player formulation still groups together vertices whose rooted neighborhoods are isomorphic, but one typically expects several distinct player types. However, this complexity disappears as \(n\to\infty\). Indeed, the annealed laws of the random unimodular measures \(\{\mu_{R_n}\}_{n\in\N}\) converge weakly to the unimodular measure concentrated on the infinite \(D\)-regular tree \(T_D\) (see \cite{van2024random}, Theorem 2.17). Since \(T_D\) is transitive, the limiting game has only one representative player. Intuitively, this reflects the fact that cycles become invisible at bounded distances from a typical vertex as \(n\to\infty\) (see Figure~\ref{fig:d-regular}).
\begin{figure}[htb]
    \centering
    \includegraphics[width=\linewidth]{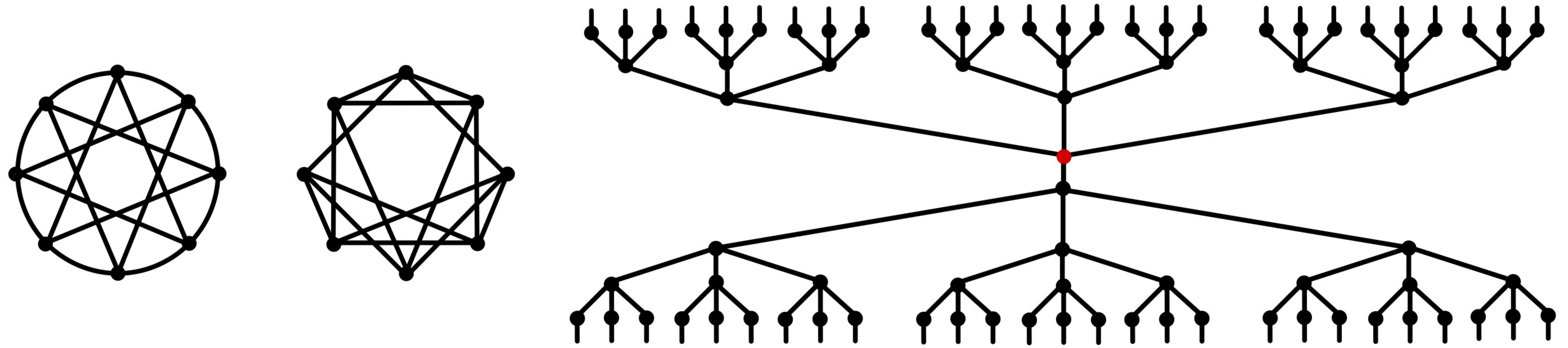}
    \caption{Illustration of the case \(D=4\). Left: a transitive realization of the random 4-regular graph $R_n$ on $[n]$ (for $n=8$). Middle: a non-transitive realization of the random 4-regular graph $R_n$ on $[n]$ (for $n=8$). Right: the local weak limit of $\{{R_n}\}_{n\in\N}$ given by the infinite 4-regular tree $T_4$ with a root (red).}
    \label{fig:d-regular}
\end{figure}
\end{example}

The above examples demonstrate the strength of the representative-player formulation. By identifying players with isomorphic rooted neighborhoods, it can reduce the dimensionality of network games substantially, often even to a single representative player. We now turn to the construction of a global potential that allows us to systematically exploit this idea. Throughout, it is helpful to keep the examples in mind as concrete illustrations.

\section{Main Results} \label{sec-results}

This section contains the main analytical results of the paper. We show that the game admits a potential functional whose first variation recovers the representative-player first-order condition. This variational structure yields a characterization of quenched Nash equilibria via minimizers of the potential and allows us to study the thermodynamic limit under weak convergence of unimodular measures.

\subsection{The potential functional}

Under Assumptions \ref{ass:degree-bound} and \ref{ass:potential}, define the potential functional $\Phi:\mathcal{H}\to\R$ by
\begin{equation} \label{eq:Phi}
\Phi(x)
:=
\int_{\mathcal G_\ast}
\bigg[
f\big([G,o],x([G,o])\big)
+
\frac12\sum_{v\sim o}\psi\big([G,o,v],x([G,o]),x([G,v])\big)
\bigg]
\mu(d[G,o]),\quad x\in\mathcal{H}.
\end{equation}
The factor \(1/2\) in \eqref{eq:Phi} compensates for the fact that each undirected edge is counted twice. The following proposition shows that $\Phi$ is well-defined.
\begin{proposition} \label{prop:Phi-finite}
Under Assumptions \ref{ass:degree-bound} and \ref{ass:potential}, the potential functional \(\Phi\) from \eqref{eq:Phi} is finite on \(\mathcal H\).
\end{proposition}

The next proposition contains an important structural observation of the framework. Owing to the potential property in Assumption~\ref{ass:potential}, the first variation of \(\Phi\) can be identified with the first-order expression of the representative (root) player in the game.

\begin{proposition}\label{prop:Phi-Gateaux}
Suppose Assumptions \ref{ass:degree-bound} and \ref{ass:potential} hold, and let \(x,{y}\in \mathcal H\). Then \(\Phi\) is G\^ateaux differentiable at \(x\), and
\begin{equation}\begin{aligned} \label{eq:first-variation}
&\langle \nabla\Phi(x),y\rangle_\mathcal{H}\\
&=
\int_{\mathcal G_\ast}
\bigg\langle
\nabla f([G,o],x([G,o]))
+
\sum_{v\sim o}\nabla_1 h([G,o,v],x([G,o]),x([G,v])),
\,{y}([G,o])
\bigg\rangle_H
\mu(d[G,o]).
\end{aligned}\end{equation}
\end{proposition}
Propositions~\ref{prop:Phi-finite} and \ref{prop:Phi-Gateaux} are proved in Section~\ref{sec:proofs-Phi}. Motivated by Proposition \ref{prop:Phi-Gateaux}, define the representative-player first-order operator \(F\) on $\mathcal{H}$ by
\begin{equation} \label{eq:F}
F(x)([G,o])
:=\nabla f([G,o],x([G,o]))+\sum_{v\sim o}\nabla_1 h([G,o,v],x([G,o]),x([G,v])).
\end{equation}
Under Assumptions \ref{ass:degree-bound} and \ref{ass:potential}, the operator \(F:\mathcal{H}\to\mathcal{H}\) is well-defined. Indeed, arguing as in the proof of Proposition \ref{prop:Phi-finite}, one can show that \(F(x)\in\mathcal H\). For a general closed, convex set \(K\), the associated representative-player first-order condition is the variational inequality
\begin{equation} \label{eq:VI}
\langle F(x), z-x\rangle_{\mathcal H}\ge 0,
\quad \text{for all } z\in\mathcal K.
\end{equation}
In the unconstrained case where \(K=H\), we have \(\mathcal K=\mathcal H\). Then \eqref{eq:VI} reduces to the unconstrained representative-player first-order condition
\be\label{eq:unconstrained}
F(x)=0.
\ee
\begin{remark}
The first-order condition \eqref{eq:VI} (resp. \eqref{eq:unconstrained}) is formulated in terms of the root player (or representative player), analogous to the representative player in mean-field games. As illustrated in Examples~\ref{ex:cycles}, \ref{ex:boxes}, and \ref{ex:d-regular}, this viewpoint can substantially reduce the dimensionality of the corresponding equilibrium conditions. As we shall see, the potential functional \(\Phi\) captures these conditions in a unified way and thereby provides a convenient tool for characterizing quenched Nash equilibria.
\end{remark}

\subsection{Minimizers and quenched Nash equilibria}

We now connect the variational problem for the potential functional \(\Phi\) to the notion of quenched Nash equilibria.

\begin{assumption}\label{ass:convex}
For \(\mu\)-a.e.~\([G,o]\) and every family \((b_v)_{v\sim o}\subset H\), the map
\[
a \mapsto f([G,o],a)+\sum_{v\sim o} h([G,o,v],a,b_v)
\]
is convex.
\end{assumption}

\begin{theorem}\label{thm:main}
Suppose Assumptions \ref{ass:degree-bound}, \ref{ass:potential}, and \ref{ass:convex} hold, and let \(x^\ast\in\mathcal K\). Consider the following statements:
\begin{enumerate}[label=(\roman*)]
\item \(x^\ast\) is a minimizer of \(\Phi\) over \(\mathcal K\).

\item \(x^\ast\) solves the representative-player first-order variational inequality \eqref{eq:VI}.

\item \(x^\ast\) is a quenched Nash equilibrium.
\end{enumerate}
Then \(\text{(i)}\implies\text{(ii)}\iff\text{(iii)}\). If, in addition, \(\Phi\) is convex on \(\mathcal K\), then also \(\text{(ii)}\implies\text{(i)}\).
\end{theorem}
Theorem~\ref{thm:main} is proved in Section~\ref{sec:proof-main}.
\begin{remark}
Theorem~\ref{thm:main} shows that minimizers of the potential functional \(\Phi\) are quenched Nash equilibria under convexity in the own action. Conversely, under the additional assumption that \(\Phi\) is convex, quenched Nash equilibria also minimize \(\Phi\). The key mechanism behind this correspondence is that the mass-transport principle \eqref{eq:MTP} and the potential property \eqref{eq:potential-id-1}--\eqref{eq:potential-id-2} identify the first variation of \(\Phi\) with the root player's first-order condition.
\end{remark}

\subsection{The thermodynamic limit}

Weak convergence of unimodular measures naturally arises from local weak convergence of graph sequences (see Remark~\ref{rem:local-weak-convergence}). It is therefore important to understand how the potential functional behaves under such convergence. In view of Theorem~\ref{thm:main}, this also provides information on the behavior of quenched Nash equilibria along locally weakly convergent graph sequences. 

Recall that $\mathcal{G}_\ast$ (resp.~$\mathcal{G}_{\ast\ast}$) is the space of isomorphism classes of countable, connected, rooted (resp.~doubly rooted) graphs equipped with the topology induced by $d_\ast$ (resp.~$d_{\ast\ast}$).

\begin{definition} 
Let $\{\mu_n\}_{n\in\N}$ be a sequence of probability measures on $\mathcal{G}_*$, and let $\mu$ be another probability measure on $\mathcal{G}_*$. The sequence $\{\mu_n\}_{n\in\N}$ converges weakly to $\mu$, denoted
\[
\mu_n \,\xrightarrow{w}\, \mu,
\]
if for every bounded continuous function $\theta:\mathcal{G}_* \to \mathbb{R}$, we have
\[
\int_{\mathcal{G}_*} \theta([G,o]) \, \mu_n(d[G,o]) \,\longrightarrow\, \int_{\mathcal{G}_*} \theta([G,o]) \, \mu(d[G,o]), \quad \text{as } n \to \infty.
\]
\end{definition}

\begin{assumption}\label{ass:thermo}
Let \(\{\mu_n\}_{n\in\N}\) be unimodular probability measures on \(\mathcal G_\ast\) such that \(\mu_n\xrightarrow{w}\mu\), and assume that all \(\mu_n\) are supported on graphs with uniform degree bound \(D\in\mathbb N\) as in \eqref{eq:degree-bound}. Assume further that \(f:\mathcal G_\ast\times H\to\mathbb R\) and \(\psi:\mathcal G_{\ast\ast}\times H\times H\to\mathbb R\) from \eqref{eq:fh} and \eqref{eq:psi} are continuous. Finally, suppose there exist \(\eps>0\) and a reference point $a_\ast\in H$ such that
\begin{equation}\label{eq:thermo-base-f}
\sup_{n\in\N}\int_{\mathcal G_\ast}|f([G,o],a_\ast)|^{1+\eps}\mu_n(d[G,o])<\infty,
\quad
\sup_{n\in\N}\int_{\mathcal G_\ast}\|\nabla f([G,o],a_\ast)\|_H^{2+\eps}\mu_n(d[G,o])<\infty,
\end{equation}
and
\begin{equation}\label{eq:thermo-base-psi}
\sup_{n\in\N}\int_{\mathcal G_\ast}\sum_{v\sim o}|\psi([G,o,v],a_\ast,a_\ast)|^{1+\eps}\mu_n(d[G,o])<\infty,
\end{equation}
\begin{equation}\label{eq:thermo-base-h}
\sup_{n\in\N}\int_{\mathcal G_\ast}\sum_{v\sim o}\|\nabla_1 h([G,o,v],a_\ast,a_\ast)\|_H^{2+\eps}\mu_n(d[G,o])<\infty.
\end{equation}
\end{assumption}

\begin{theorem}\label{thm:thermo-limit}
Suppose Assumptions \ref{ass:degree-bound}, \ref{ass:potential}(i)--(iii), and \ref{ass:thermo} hold. Let \(x\in\mathcal H\) be an action rule satisfying
\begin{equation}\label{eq:thermo-x}
\sup_{n\in\N}\int_{\mathcal G_\ast}\|x([G,o])\|_H^{2+\eps}\mu_n(d[G,o])<\infty,\quad \int_{\mathcal G_\ast}\|x([G,o])\|_H^{2+\eps}\mu(d[G,o])<\infty,
\end{equation}
for the same \(\eps>0\) as in Assumption~\ref{ass:thermo}. Then
\[
\Phi_n(x)\longrightarrow \Phi(x),
\quad \text{as } n \to \infty,
\]
where \(\Phi_n\) denotes the potential functional \eqref{eq:Phi} with \(\mu\) replaced by \(\mu_n\).
\end{theorem}
Theorem~\ref{thm:thermo-limit} is proved in Section~\ref{sec:proof-conv}.
\begin{remark}
The convergence result in Theorem~\ref{thm:thermo-limit} applies to Examples~\ref{ex:cycles}, \ref{ex:boxes}, and \ref{ex:d-regular}, namely cycle graphs, boxes in \(\mathbb Z^d\), and random \(D\)-regular graphs. For the latter two examples, the local weak limit has a substantially simpler representative-player structure than the corresponding finite graphs, and in fact reduces to a single representative player. This shows that the thermodynamic limit may lead to a significant reduction in dimensionality when analyzing the corresponding potential functional and equilibrium conditions.
\end{remark}

\section{Examples}\label{sec:examples}

We first consider a static linear–quadratic game, which fits our framework and leads to a linear Poisson equation for quenched Nash equilibria. Under an assumption on the simple random walk operator associated with the graphs in the support of $\mu$, this allows the derivation of an explicit solution via the corresponding Green kernel. We then turn to a more general nonlinear potential game, where we characterize the equilibrium through a nonlinear Poisson equation.

\subsection{Linear--quadratic potential games and random walks on graphs}\label{sec:lq-example}

Assume that $\mu$ is a unimodular measure with \(\deg_G(o)>0\) for \(\mu\)-a.e.~$[G,o]$, such that Assumption~\ref{ass:degree-bound} holds. Let \(H=K=\mathbb R\), and let \(\eta\in L^2(\mu,\R)\). Define $f$, $h$, and $\psi$ from \eqref{eq:J} and \eqref{eq:psi} as follows, 
\begin{equation}\label{eq:lq-f}
f:\mathcal{G}_\ast\times\R\to\R,\quad f([G,o],a):=-\eta([G,o])\,a,
\end{equation}
and
\begin{equation}\label{eq:lq-h}
h:\mathcal{G}_{\ast\ast}\times\R\times\R\to\R,\quad h([G,o,v],a,b):=\frac12(a-b)^2.
\end{equation}
Set also
\begin{equation}\label{eq:lq-psi}
\psi:=h.
\end{equation}
Define the simple random walk operator $P$ on $\mathcal H$ by
\begin{equation}\label{eq:lq-P}
(Px)([G,o])
:=
\frac{1}{\deg_G(o)}\sum_{v\sim o}x([G,v]).
\end{equation}
Define the associated Laplacian by
\begin{equation}\label{eq:lq-Delta}
\Delta:=I-P,
\end{equation}
where \(I\) denotes the identity operator on \(\mathcal H\). 

\begin{proposition}\label{prop:lq-example}
For the choice \eqref{eq:lq-f}--\eqref{eq:lq-psi}, Assumptions \ref{ass:potential} and \ref{ass:convex} are satisfied. The corresponding potential functional is given by
\begin{equation}\label{eq:lq-potential}
\Phi(x)
=
\int_{\mathcal G_\ast}
\bigg[
-\eta([G,o])\,x([G,o])
+
\frac14\sum_{v\sim o}\big(x([G,o])-x([G,v])\big)^2
\bigg]\mu(d[G,o]),
\quad x\in\mathcal H.
\end{equation}
Consequently, if \(x^\ast\in\mathcal H\) satisfies
\begin{equation}\label{eq:lq-poisson}
(\Delta x^\ast)([G,o])=\frac{\eta([G,o])}{\deg_G(o)},
\quad \text{for } \mu\text{-a.e. } [G,o],
\end{equation}
then \(x^\ast\) is a quenched Nash equilibrium.
\end{proposition}
Proposition~\ref{prop:lq-example} is proved in Section~\ref{sec:proofs-examples}. Next, we study under which conditions the Poisson equation \eqref{eq:lq-poisson} admits an explicit solution. Fix a graph \(G\) in the support of \(\mu\). Consider the simple random walk on $G$ given by the Markov chain with state space $V(G)$ and transition probabilities
\be
p_G(v,u)=
\begin{cases}
    \frac{1}{\deg_G(v)} &\text{if }u\sim v,\\
    0 & \text{otherwise,}
\end{cases}
\quad v,u\in V(G)
\ee
(see \cite{woess2000random}, Chapter I, Section 1.C, Equation 1.19). Since $G$ is undirected and locally finite, the simple random walk is reversible with respect to the measure
\begin{equation}\label{eq:lq-mG}
m_G:V(G)\to(0,\infty),\quad m_G(v):=\deg_G(v),
\end{equation}
that is,
$$
m_G(v)p_G(v,u)=m_G(u)p_G(u,v),\quad\text{for all }v,u\in V(G)
$$
(see \cite{woess2000random}, Chapter I, Section 2.A, Equation 2.1). Let \(P_G=(p_G(v,u))_{v,u\in V(G)}\) denote the associated simple random walk operator on \(\ell^2(V(G),m_G)\), defined by
\begin{equation}\label{eq:lq-PG}
(P_G \xi)(v)
:=
\frac{1}{\deg_G(v)}\sum_{u\sim v}\xi(u),
\quad v\in V(G).
\end{equation}
Define the corresponding Laplacian on \(\ell^2(V(G),m_G)\) by
\begin{equation}\label{eq:lq-DeltaG}
\Delta_G:=I_G-P_G,
\end{equation}
where $I_G$ denotes the identity operator on \(\ell^2(V(G),m_G)\). Moreover, we define the Green kernel of the simple random walk by
\begin{equation}\label{eq:lq-green-kernel}
K_G(v,u)
:=
\sum_{n=0}^\infty p_G^{(n)}(v,u),
\quad v,u\in V(G),
\end{equation}
where \(p_G^{(n)}(v,u)\) denotes the probability that the simple random walk starting at $v$ arrives at $u$ after $n$ steps (see \cite{woess2000random}, Chapter I, Section 1.B, Equation 1.6). We note that the infinite series in \eqref{eq:lq-green-kernel} either converges or diverges simultaneously for all $v,u\in V(G)$ (see \cite{woess2000random}, Chapter I, Section 1.B, Lemma 1.7). Thus, we can define the so-called spectral radius of the simple random walk operator $P_G$ from \eqref{eq:lq-PG} as
\be\label{eq:lq-spectral-radius-def}
\rho(P_G):=\limsup_{n\to\infty}p_G^{(n)}(v,u)^{1/n}\in (0,1],
\ee
independent of the choice of $v,u\in V(G)$ (see \cite{woess2000random}, Chapter I, Section 1.B, Equation 1.8).

\begin{corollary}\label{cor:green}
Assume that for \(\mu\)-a.e.~\([G,o]\), the simple random walk operator \(P_G\) satisfies
\be\label{eq:lq-spectral-radius}
\rho(P_G)<1,
\ee
and that $\xi^G$ defined by
\be\label{eq:xi}
\xi^G_v:=\frac{\eta([G,v])}{\deg_G(v)},\quad v\in V(G),
\ee
belongs to \(\ell^2(V(G),m_G)\). Then the Green kernel \(K_G\) from \eqref{eq:lq-green-kernel} defines a bounded linear operator on \(\ell^2(V(G),m_G)\), and if \(x^\ast\in\mathcal H\) is given by \eqref{eq:xGv} with
\begin{equation}\label{eq:lq-green}
x^{\ast,G}_v
=
\sum_{u\in V(G)}K_G(v,u)\,\xi^G_u,
\quad v\in V(G),
\end{equation}
then \(x^\ast\) is a quenched Nash equilibrium.
\end{corollary}
Corollary~\ref{cor:green} is proved in Section~\ref{sec:proofs-examples}.
\begin{remark}
Conditions for a given graph $G$ such that its spectral radius satisfies \(\rho(P_G)<1\)  as in \eqref{eq:lq-spectral-radius} have been studied extensively in the literature on random walks. A necessary (but not sufficient) condition for \eqref{eq:lq-spectral-radius} is that $G$ is infinite and has exponential growth (see \cite{woess2000random}, Chapter~II, Section~10.B). In particular, \eqref{eq:lq-spectral-radius} fails for graphs of subexponential growth, such as P\'olya's walk on $\mathbb{Z}^d$ (see \cite{woess2000random}, Chapter~I, Section~1.A). On the other hand, there are large classes of infinite graphs for which \eqref{eq:lq-spectral-radius} does hold. For instance, any locally finite tree with minimum degree \(2\) satisfies the spectral radius condition whenever the lengths of its unbranched paths are uniformly bounded (see \cite{woess2000random}, Chapter II, Section 10.C, Theorem~10.9). In particular, the infinite $D$-regular tree from Example~\ref{ex:d-regular} satisfies \eqref{eq:lq-spectral-radius} whenever $D\geq 3$. See also Section~10.C in Chapter~II of \cite{woess2000random} for further examples.
\end{remark}

\subsection{Nonlinear potential games}

We now generalize the previous example to a nonlinear potential game. Assume that \(\mu\) is a unimodular measure with \(\deg_G(o)>0\) for \(\mu\)-a.e.~$[G,o]$, and that Assumption~\ref{ass:degree-bound} holds. Let \(H=K=\mathbb R\), and let \(\eta\in L^2(\mu,\mathbb R)\). Let \(W:\mathbb R\to\mathbb R\) be a convex, even, continuously differentiable function. Define $f,h$ and $\psi$ from \eqref{eq:J}, and \eqref{eq:psi} as follows, 
\begin{equation}\label{eq:W-f}
f:\mathcal G_\ast\times\mathbb R\to\mathbb R,
\quad
f([G,o],a):=-\eta([G,o])\,a,
\end{equation}
and
\begin{equation}\label{eq:W-h}
h:\mathcal G_{\ast\ast}\times\mathbb R\times\mathbb R\to\mathbb R,
\quad
h([G,o,v],a,b):=W(a-b),
\end{equation}
with $\psi:=h$.
Define the nonlinear Laplacian \(\Delta_W\) on \(\mathcal H\) by 
\begin{equation}\label{eq:W-Delta}
(\Delta_W x)([G,o])
:=
\frac1{\deg_G(o)}
\sum_{v\sim o}W'\!\big(x([G,o])-x([G,v])\big).
\end{equation}
\begin{remark}
When moving from the linear-quadratic to the nonlinear formulation, Assumption~\ref{ass:potential}(i), Assumption~\ref{ass:potential}(ii), Assumption~\ref{ass:potential}(iv), and Assumption~\ref{ass:convex} remain satisfied. In particular, the evenness of $W$ ensures that the potential identities \eqref{eq:potential-id-1}--\eqref{eq:potential-id-2} hold with $\psi=h$. However, the global Lipschitz condition \eqref{eq:lip-h} is satisfied if and only if \(W'\) is globally Lipschitz on \(\mathbb R\).
\end{remark}
The following result generalizes Proposition~\ref{prop:lq-example} to the nonlinear case.
\begin{proposition}\label{prop:W-example}
If \(x^\ast\in\mathcal H\) satisfies
\begin{equation}\label{eq:W-poisson}
(\Delta_W x^\ast)([G,o])=\frac{\eta([G,o])}{\deg_G(o)},
\quad \text{for } \mu\text{-a.e. } [G,o],
\end{equation}
then \(x^\ast\) is a quenched Nash equilibrium.
\end{proposition}
Proposition~\ref{prop:W-example} is proved in Section~\ref{sec:proofs-examples}.
\begin{remark}
The choice
\[
W(s)=\frac1p|s|^p,
\quad p>1,
\]
yields a \(p\)-Laplacian potential game. In this case, the nonlinear Laplacian \eqref{eq:W-Delta} becomes
\[
(\Delta_p x)([G,o])
:=
\frac1{\deg_G(o)}
\sum_{v\sim o}
|x([G,o])-x([G,v])|^{p-2}\big(x([G,o])-x([G,v])\big),
\]
and the equilibrium can be characterized through the nonlinear Poisson equation
\be\label{eq:lq-poisson-nonlinear}
(\Delta_p x^\ast)([G,o])
=
\frac{\eta([G,o])}{\deg_G(o)},
\quad \text{for } \mu\text{-a.e. } [G,o].
\ee
If \(p=2\), \eqref{eq:lq-poisson-nonlinear} reduces exactly to the linear Poisson equation \eqref{eq:lq-poisson} from the linear--quadratic example in Section~\ref{sec:lq-example}.
\end{remark}

\section{Proofs of Propositions~\ref{prop:Phi-finite} and \ref{prop:Phi-Gateaux}}\label{sec:proofs-Phi}

This section is devoted to the proofs of the results regarding the well-definedness and first variation of the potential functional $\Phi$ from \eqref{eq:Phi}.

\begin{proof}[Proof of Proposition~\ref{prop:Phi-finite}]
Recalling \eqref{eq:H}, fix \(x\in\mathcal H\), and write
\[
x_o:=x([G,o]), \quad x_v:=x([G,v]).
\]
Let $a_\ast$ be as in Assumption \ref{ass:potential}(iv). First, we prove that the first term in \eqref{eq:Phi} is integrable. Consider the path given by
\[
\gamma:[0,1]\to H,\quad \gamma(t):=a_\ast+t(x_o-a_\ast),
\]
and define the composition $g(t):=f([G,o],\gamma(t))$. Since $a\mapsto f([G,o],a)$ is G\^ateaux differentiable by Assumption~\ref{ass:potential}(i), we have for every \(t\in(0,1)\),
\be\begin{aligned}\label{eq:curve-derivative}
g'(t)
&=
\lim_{s\to 0}\frac{g(t+s)-g(t)}{s} \\
&=
\lim_{s\to 0}\frac{
f([G,o],\gamma(t+s))-f([G,o],\gamma(t))
}{s} \\
&=
\lim_{s\to 0}\frac{
f([G,o],\gamma(t)+s(x_o-a_\ast))-f([G,o],\gamma(t))
}{s} \\
&=
\left\langle
\nabla f([G,o],\gamma(t)),\,x_o-a_\ast
\right\rangle_H.
\end{aligned}\ee
The fundamental theorem of calculus therefore yields
\be\label{eq:ftc}
f([G,o],x_o)-f([G,o],a_\ast)
=
\int_0^1
\left\langle
\nabla f\big([G,o],a_\ast+t(x_o-a_\ast)\big),\,x_o-a_\ast
\right\rangle_Hdt.
\ee
Hence, by the Cauchy--Schwarz inequality,
\begin{equation} \label{eq:f-bound-start}
|f([G,o],x_o)|
\le
|f([G,o],a_\ast)|
+
\int_0^1
\left\|
\nabla f\big([G,o],a_\ast+t(x_o-a_\ast)\big)
\right\|_H
\|x_o-a_\ast\|_Hdt.
\end{equation}
Next, by \eqref{eq:lip-f} it holds that
\be\label{eq:nabla-f-bound}
\left\|
\nabla f\big([G,o],a_\ast+t(x_o-a_\ast)\big)
\right\|_H
\le
L_f\,t\|x_o-a_\ast\|_H+\|\nabla f([G,o],a_\ast)\|_H.
\ee
Substituting \eqref{eq:nabla-f-bound} into \eqref{eq:f-bound-start} yields
\begin{equation} \label{eq:f-bound}
|f([G,o],x_o)|
\le
|f([G,o],a_\ast)|
+
\frac{L_f}{2}\|x_o-a_\ast\|_H^2
+
\|\nabla f([G,o],a_\ast)\|_H\,\|x_o-a_\ast\|_H.
\end{equation}
Integrating \eqref{eq:f-bound} with respect to $\mu$, we obtain
\be\begin{aligned}\label{eq:f-L1}
\int_{\mathcal G_\ast}|f([G,o],x_o)|\mu(d[G,o])
&\le
\int_{\mathcal G_\ast}|f([G,o],a_\ast)|\mu(d[G,o])
+
\frac{L_f}{2}\int_{\mathcal G_\ast}\|x_o-a_\ast\|_H^2\mu(d[G,o])
 \\
&\quad
+
\int_{\mathcal G_\ast}\|\nabla f([G,o],a_\ast)\|_H\,\|x_o-a_\ast\|_H\mu(d[G,o]).
\end{aligned}\ee
The first term on the right-hand side of \eqref{eq:f-L1} is finite by \eqref{eq:int-base-f}. The second term is finite since \(x\in\mathcal H\). Finally, the third term is finite by the Cauchy--Schwarz inequality and \eqref{eq:int-base-f}. Hence
\begin{equation} \label{eq:f-int-finite}
\int_{\mathcal G_\ast}|f([G,o],x_o)|\mu(d[G,o])<\infty.
\end{equation}
Next, we prove that the second term on the right-hand side of \eqref{eq:Phi} is integrable. By \eqref{eq:potential-id-1}, for any \(a,b\in H\),
\[
\|\nabla_1\psi([G,o,v],a,b)\|_H
=
\|\nabla_1 h([G,o,v],a,b)\|_H.
\]
Using \eqref{eq:lip-h}, we obtain
\be\begin{aligned}\label{eq:nabla-psi-bound1}
&\|\nabla_1\psi([G,o,v],a,b)\|_H\\
&\le
\|\nabla_1 h([G,o,v],a,b)-\nabla_1 h([G,o,v],a_\ast,a_\ast)\|_H
+
\|\nabla_1 h([G,o,v],a_\ast,a_\ast)\|_H \\
&\le
L_h\big(\|a-a_\ast\|_H+\|b-a_\ast\|_H\big)
+
\|\nabla_1 h([G,o,v],a_\ast,a_\ast)\|_H.
\end{aligned}\ee
Similarly, by \eqref{eq:potential-id-2} and \eqref{eq:lip-h},
\be\begin{aligned}\label{eq:nabla-psi-bound2}
\|\nabla_2\psi([G,o,v],a,b)\|_H
&=
\|\nabla_1 h([G,v,o],b,a)\|_H \\
&\le
L_h\big(\|a-a_\ast\|_H+\|b-a_\ast\|_H\big)
+
\|\nabla_1 h([G,v,o],a_\ast,a_\ast)\|_H.
\end{aligned}\ee
Now consider the path given by
\[
\gamma:[0,1]\to H\times H,\quad\gamma(t):=\big(a_\ast+t(x_o-a_\ast),\,a_\ast+t(x_v-a_\ast)\big).
\]
Arguing exactly as in \eqref{eq:curve-derivative} and \eqref{eq:ftc}, we obtain from the fundamental theorem of calculus,
\be\begin{aligned}
&\psi([G,o,v],x_o,x_v)-\psi([G,o,v],a_\ast,a_\ast) \\
&=
\int_0^1
\left\langle
\nabla_1\psi\big([G,o,v],\gamma(t)\big),\,x_o-a_\ast
\right\rangle_Hdt
+
\int_0^1
\left\langle
\nabla_2\psi\big([G,o,v],\gamma(t)\big),\,x_v-a_\ast
\right\rangle_Hdt.
\end{aligned}\ee
Hence, by the Cauchy-Schwarz inequality,
\be\begin{aligned}\label{eq:psi-bound-start}
|\psi([G,o,v],x_o,x_v)|
&\le
|\psi([G,o,v],a_\ast,a_\ast)|
+
\int_0^1
\|\nabla_1\psi([G,o,v],\gamma(t))\|_H\,\|x_o-a_\ast\|_Hdt
\\
&\quad
+
\int_0^1
\|\nabla_2\psi([G,o,v],\gamma(t))\|_H\,\|x_v-a_\ast\|_Hdt.
\end{aligned}\ee
Using \eqref{eq:nabla-psi-bound1} and \eqref{eq:nabla-psi-bound2}, we get from \eqref{eq:psi-bound-start},
\be\begin{aligned}\label{eq:psi-bound}
|\psi([G,o,v],x_o,x_v)|
&\le
|\psi([G,o,v],a_\ast,a_\ast)|
+
L_h\big(\|x_o-a_\ast\|_H+\|x_v-a_\ast\|_H\big)\|x_o-a_\ast\|_H
 \\
&\quad
+
L_h\big(\|x_o-a_\ast\|_H+\|x_v-a_\ast\|_H\big)\|x_v-a_\ast\|_H
 \\
&\quad
+
\|\nabla_1 h([G,o,v],a_\ast,a_\ast)\|_H\,\|x_o-a_\ast\|_H
 \\
&\quad
+
\|\nabla_1 h([G,v,o],a_\ast,a_\ast)\|_H\,\|x_v-a_\ast\|_H\\
&\le
|\psi([G,o,v],a_\ast,a_\ast)|
+
2L_h\|x_o-a_\ast\|_H^2
+
2L_h\|x_v-a_\ast\|_H^2
 \\
&\quad
+
\|\nabla_1 h([G,o,v],a_\ast,a_\ast)\|_H\,\|x_o-a_\ast\|_H
 \\
&\quad
+
\|\nabla_1 h([G,v,o],a_\ast,a_\ast)\|_H\,\|x_v-a_\ast\|_H.
\end{aligned}\ee
Summing \eqref{eq:psi-bound} over \(v\sim o\) and integrating with respect to $\mu$, we obtain
\be\begin{aligned}\label{eq:psi-L1}
&\int_{\mathcal G_\ast}\sum_{v\sim o}|\psi([G,o,v],x_o,x_v)|\mu(d[G,o])
 \\
&\le
\int_{\mathcal G_\ast}\sum_{v\sim o}|\psi([G,o,v],a_\ast,a_\ast)|\mu(d[G,o])
 \\
&\quad
+
2L_h\int_{\mathcal G_\ast}\sum_{v\sim o}\|x_o-a_\ast\|_H^2\mu(d[G,o])
+
2L_h\int_{\mathcal G_\ast}\sum_{v\sim o}\|x_v-a_\ast\|_H^2\mu(d[G,o])
 \\
&\quad
+
\int_{\mathcal G_\ast}\sum_{v\sim o}\|\nabla_1 h([G,o,v],a_\ast,a_\ast)\|_H\,\|x_o-a_\ast\|_H\mu(d[G,o])
 \\
&\quad
+
\int_{\mathcal G_\ast}\sum_{v\sim o}\|\nabla_1 h([G,v,o],a_\ast,a_\ast)\|_H\,\|x_v-a_\ast\|_H\mu(d[G,o]).
\end{aligned}\ee
The first term on the right-hand side of \eqref{eq:psi-L1} is finite by \eqref{eq:int-base-psi}. For the second term, the degree bound \eqref{eq:degree-bound} and the fact that $x\in\mathcal{H}$ yield
\be\label{eq:second-term}
\int_{\mathcal G_\ast}\sum_{v\sim o}\|x_o-a_\ast\|_H^2\mu(d[G,o])
\le
D\int_{\mathcal G_\ast}\|x_o-a_\ast\|_H^2\mu(d[G,o])<\infty.
\ee
For the third term, the mass-transport principle \eqref{eq:MTP} applied to the transport function $T([G,o,v])=\mathbf 1_{\{o\sim v\}}\|x_o-a_\ast\|_H^2$ and \eqref{eq:second-term} give
\[
\int_{\mathcal G_\ast}\sum_{v\sim o}\|x_v-a_\ast\|_H^2\mu(d[G,o])
=
\int_{\mathcal G_\ast}\sum_{v\sim o}\|x_o-a_\ast\|_H^2\mu(d[G,o])<\infty.
\]
For the fourth term, the Cauchy--Schwarz inequality yields
\be\begin{aligned}\label{eq:fourth-term}
&\int_{\mathcal G_\ast}\sum_{v\sim o}\|\nabla_1 h([G,o,v],a_\ast,a_\ast)\|_H\,\|x_o-a_\ast\|_H\mu(d[G,o]) \\
&\le
\bigg(
\int_{\mathcal G_\ast}\sum_{v\sim o}\|\nabla_1 h([G,o,v],a_\ast,a_\ast)\|_H^2\mu(d[G,o])
\bigg)^{1/2}
\bigg(
\int_{\mathcal G_\ast}\sum_{v\sim o}\|x_o-a_\ast\|_H^2\mu(d[G,o])
\bigg)^{1/2},
\end{aligned}\ee
which is finite by \eqref{eq:int-base-h} and \eqref{eq:second-term}. For the last term, the mass-transport principle \eqref{eq:MTP} and \eqref{eq:int-base-h} give
\[
\int_{\mathcal G_\ast}\sum_{v\sim o}\|\nabla_1 h([G,v,o],a_\ast,a_\ast)\|_H^2\mu(d[G,o])
\hspace{-0.6mm}=\hspace{-0.7mm}
\int_{\mathcal G_\ast}\sum_{v\sim o}\|\nabla_1 h([G,o,v],a_\ast,a_\ast)\|_H^2\mu(d[G,o])<\infty,
\]
and therefore the same Cauchy--Schwarz argument as in \eqref{eq:fourth-term} applies. Hence
\begin{equation} \label{eq:psi-int-finite}
\int_{\mathcal G_\ast}\sum_{v\sim o}|\psi([G,o,v],x_o,x_v)|\mu(d[G,o])<\infty.
\end{equation}
Combining \eqref{eq:f-int-finite} and \eqref{eq:psi-int-finite}, we conclude that
\[
\int_{\mathcal G_\ast}
\Big|
f([G,o],x_o)
+
\frac12\sum_{v\sim o}\psi([G,o,v],x_o,x_v)
\Big|
\mu(d[G,o])<\infty.
\]
Hence the integrand in \eqref{eq:Phi} is absolutely integrable, and therefore \(\Phi(x)\in\mathbb R\).
\end{proof}

We continue with the proof of Proposition~\ref{prop:Phi-Gateaux}, where we compute the G\^ateaux derivative of $\Phi$.

\begin{proof}[Proof of Proposition~\ref{prop:Phi-Gateaux}]
Recalling \eqref{eq:H}, fix \(x,{y}\in\mathcal H\), and write
\[
x_o:=x([G,o]), \quad x_v:=x([G,v]), \quad {y}_o:={y}([G,o]), \quad {y}_v:={y}([G,v]).
\]
For $t\in\R$, define $\Xi:\mathcal{G}_\ast\times\R\to\R$ by
\be\label{eq:Xi}
\Xi_t([G,o])
:=
f([G,o],x_o+t{y}_o)
+
\frac12\sum_{v\sim o}\psi([G,o,v],x_o+t{y}_o,x_v+t{y}_v).
\ee
Then, by definition of $\Phi$ in \eqref{eq:Phi},
\be\label{eq:Phi-Xi}
\Phi(x+t{y})=\int_{\mathcal G_\ast}\Xi_t([G,o])\mu(d[G,o]).
\ee
Since \(a\mapsto f([G,o],a)\) and \((a,b)\mapsto \psi([G,o,v],a,b)\) are G\^ateaux differentiable by Assumption \ref{ass:potential}, for every fixed \([G,o]\) we have by \eqref{eq:Xi} and a similar argument as in \eqref{eq:curve-derivative},
\be\begin{aligned}\label{eq:Xi'}
\frac{d}{dt}\Xi_t([G,o])
&=
\left\langle \nabla f([G,o],x_o+t{y}_o),{y}_o\right\rangle_H \\
&\quad
+
\frac12\sum_{v\sim o}
\left\langle \nabla_1\psi([G,o,v],x_o+t{y}_o,x_v+t{y}_v),{y}_o\right\rangle_H \\
&\quad
+
\frac12\sum_{v\sim o}
\left\langle \nabla_2\psi([G,o,v],x_o+t{y}_o,x_v+t{y}_v),{y}_v\right\rangle_H .
\end{aligned}\ee
Next, by \eqref{eq:lip-f} it holds for any $a\in H$ that
\be\label{eq:first-var-f-growth}
\|\nabla f([G,o],a)\|_H
\le
L_f\|a-a_\ast\|_H+\|\nabla f([G,o],a_\ast)\|_H.
\ee
 Likewise, by \eqref{eq:potential-id-1} and \eqref{eq:lip-h}, it holds for any $a,b\in H$ that
\be\begin{aligned}
\|\nabla_1\psi([G,o,v],a,b)\|_H
&=
\|\nabla_1 h([G,o,v],a,b)\|_H
 \\
&\le
L_h\big(\|a-a_\ast\|_H+\|b-a_\ast\|_H\big)
+
\|\nabla_1 h([G,o,v],a_\ast,a_\ast)\|_H,
\label{eq:first-var-psi1-growth}
\end{aligned}\ee
and similarly, by \eqref{eq:potential-id-2} and \eqref{eq:lip-h},
\be\begin{aligned}
\|\nabla_2\psi([G,o,v],a,b)\|_H
&=
\|\nabla_1 h([G,v,o],b,a)\|_H
 \\
&\le
L_h\big(\|a-a_\ast\|_H+\|b-a_\ast\|_H\big)
+
\|\nabla_1 h([G,v,o],a_\ast,a_\ast)\|_H.
\label{eq:first-var-psi2-growth}
\end{aligned}\ee
Assuming $|t|\leq 1$, it holds that
\[
\|x_o+t{y}_o\|_H \le \|x_o\|_H+\|{y}_o\|_H,
\quad
\|x_v+t{y}_v\|_H \le \|x_v\|_H+\|{y}_v\|_H.
\]
Therefore, using \eqref{eq:Xi'}, \eqref{eq:first-var-f-growth}, \eqref{eq:first-var-psi1-growth}, \eqref{eq:first-var-psi2-growth}, and the Cauchy-Schwarz inequality, we obtain 
\be\begin{aligned}\label{eq:first-var-dom-1}
\bigg|\frac{d}{dt}\Xi_t([G,o])\bigg|
&\le
C\big(1+\|x_o\|_H+\|{y}_o\|_H\big)\|{y}_o\|_H
+
\|\nabla f([G,o],a_\ast)\|_H\,\|{y}_o\|_H
\\
&\quad
+
C\sum_{v\sim o}
\big(1+\|x_o\|_H+\|x_v\|_H+\|{y}_o\|_H+\|{y}_v\|_H\big)\|{y}_o\|_H
\\
&\quad
+
\frac12\sum_{v\sim o}\|\nabla_1 h([G,o,v],a_\ast,a_\ast)\|_H\,\|{y}_o\|_H
\\
&\quad
+
C\sum_{v\sim o}
\big(1+\|x_o\|_H+\|x_v\|_H+\|{y}_o\|_H+\|{y}_v\|_H\big)\|{y}_v\|_H
\\
&\quad
+
\frac12\sum_{v\sim o}\|\nabla_1 h([G,v,o],a_\ast,a_\ast)\|_H\,\|{y}_v\|_H.
\end{aligned}\ee
for a positive constant \(C<\infty\) not depending on $|t|\leq 1$. Using Young's inequality and the degree bound \eqref{eq:degree-bound}, it follows from \eqref{eq:first-var-dom-1} that
\begin{align}
\left|\frac{d}{dt}\Xi_t([G,o])\right|
&\le
C'\bigg(
1+\|x_o\|_H^2+\|{y}_o\|_H^2+\sum_{v\sim o}\|x_v\|_H^2+\sum_{v\sim o}\|{y}_v\|_H^2
\bigg) + 2\|\nabla f([G,o],a_\ast)\|_H^2\\
&\quad
+
\sum_{v\sim o}\|\nabla_1 h([G,o,v],a_\ast,a_\ast)\|_H^2
+
\sum_{v\sim o}\|\nabla_1 h([G,v,o],a_\ast,a_\ast)\|_H^2
\label{eq:first-var-dom-2}
\end{align}
for a positive constant \(C'<\infty\) not depending on $|t|\leq 1$. We claim that the right-hand side of \eqref{eq:first-var-dom-2} is integrable with respect to \(\mu\). Indeed, \(x,{y}\in\mathcal H\) and by the mass-transport principle \eqref{eq:MTP},
\[
\int_{\mathcal G_\ast}\sum_{v\sim o}\|x_v\|_H^2\mu(d[G,o])
=
\int_{\mathcal G_\ast}\sum_{v\sim o}\|x_o\|_H^2\mu(d[G,o])
\le
D\|x\|_{\mathcal H}^2<\infty,
\]
and similarly
\[
\int_{\mathcal G_\ast}\sum_{v\sim o}\|{y}_v\|_H^2\mu(d[G,o])
=
\int_{\mathcal G_\ast}\sum_{v\sim o}\|y_o\|_H^2\mu(d[G,o])
\le
D\|{y}\|_{\mathcal H}^2<\infty,
\]
thus the first term on the right-hand side of \eqref{eq:first-var-dom-2} is integrable. For the other three terms, \eqref{eq:int-base-f}, \eqref{eq:int-base-h} and a further application of the mass-transport principle \eqref{eq:MTP} give
\[
\int_{\mathcal G_\ast}\|\nabla f([G,o],a_\ast)\|_H^2\mu(d[G,o])<\infty
\]
and
\[
\int_{\mathcal G_\ast}\sum_{v\sim o}\|\nabla_1 h([G,v,o],a_\ast,a_\ast)\|_H^2\mu(d[G,o])\hspace{-0.6mm}=\hspace{-0.7mm}\int_{\mathcal G_\ast}\sum_{v\sim o}\|\nabla_1 h([G,o,v],a_\ast,a_\ast)\|_H^2\mu(d[G,o])<\infty.
\]
Hence the right-hand side of \eqref{eq:first-var-dom-2} is integrable. 
Define the difference quotient
\be\label{eq:Qt}
Q_t([G,o])
:=
\frac{\Xi_t([G,o])-\Xi_0([G,o])}{t},
\quad t \in [-1,1] \setminus \{0\}.
\ee
Since, for every fixed \([G,o]\), the map \(t\mapsto \Xi_t([G,o])\) is differentiable by \eqref{eq:Xi'}, the fundamental theorem of calculus yields
\[
Q_t([G,o])
=
\frac{1}{t}\int_0^t \frac{d}{dr}\Xi_{r}([G,o])dr.
\]
Hence, by \eqref{eq:first-var-dom-2}, 
\[
|Q_t([G,o])|
\le
\sup_{|r|\le 1}\left|\frac{d}{dr}\Xi_r([G,o])\right|
\le
M([G,o]),
\quad \textrm{for all } t \in [-1,1] \setminus \{0\},
\]
where \(M([G,o])\) denotes the $\mu$-integrable right-hand side of \eqref{eq:first-var-dom-2}. 
Therefore, by the dominated convergence theorem,
\be\label{eq:dct}
\lim_{t\to 0}
\int_{\mathcal G_\ast} Q_t([G,o])\mu(d[G,o])
=
\int_{\mathcal G_\ast}
\frac{d}{dt}\Xi_t([G,o])\bigg|_{t=0}\,
\mu(d[G,o]).
\ee
Since by \eqref{eq:Phi-Xi} and \eqref{eq:Qt},
\[
\int_{\mathcal G_\ast} Q_t([G,o])\mu(d[G,o])
=
\frac{\Phi(x+ty)-\Phi(x)}{t},
\]
\eqref{eq:dct} shows that differentiation under the integral sign for $\Phi$ in \eqref{eq:Phi} is justified. Consequently, by \eqref{eq:potential-id-1}, \eqref{eq:potential-id-2}, and \eqref{eq:Xi'},
\be\begin{aligned}
\langle \nabla\Phi(x),y\rangle_\mathcal{H}
&=
\int_{\mathcal G_\ast}
\left\langle \nabla f([G,o],x_o),{y}_o\right\rangle_H
\mu(d[G,o])
 \\
&\quad
+
\frac12
\int_{\mathcal G_\ast}
\sum_{v\sim o}
\left\langle \nabla_1\psi([G,o,v],x_o,x_v),{y}_o\right\rangle_H
\mu(d[G,o])
 \\
&\quad
+
\frac12
\int_{\mathcal G_\ast}
\sum_{v\sim o}
\left\langle \nabla_2\psi([G,o,v],x_o,x_v),{y}_v\right\rangle_H
\mu(d[G,o]).
\\
&=
\int_{\mathcal G_\ast}
\left\langle \nabla f([G,o],x_o),{y}_o\right\rangle_H
\mu(d[G,o])
 \\
&\quad
+
\frac12
\int_{\mathcal G_\ast}
\sum_{v\sim o}
\left\langle \nabla_1 h([G,o,v],x_o,x_v),{y}_o\right\rangle_H
\mu(d[G,o])
 \\
&\quad
+
\frac12
\int_{\mathcal G_\ast}
\sum_{v\sim o}
\left\langle \nabla_1 h([G,v,o],x_v,x_o),{y}_v\right\rangle_H
\mu(d[G,o]).
\label{eq:proof-first-var-2}
\end{aligned}\ee
Now define the mass transport
\be\label{eq:T}
T([G,o,v])
:=
\mathbf 1_{\{o\sim v\}}
\left\langle \nabla_1 h([G,o,v],x_o,x_v),{y}_o\right\rangle_H.
\ee
By \eqref{eq:first-var-psi1-growth} and the Cauchy-Schwarz inequality,
\[
|T([G,o,v])|
\le
\mathbf 1_{\{o\sim v\}}
\Big(
L_h(\|x_o-a_\ast\|_H+\|x_v-a_\ast\|_H)
+
\|\nabla_1 h([G,o,v],a_\ast,a_\ast)\|_H
\Big)\|{y}_o\|_H.
\]
Using $x,y\in\mathcal{H}$, the degree bound \eqref{eq:degree-bound}, \eqref{eq:int-base-h}, and the Cauchy-Schwarz inequality, we conclude that
\be
\int_{\mathcal{G}_\ast}\sum_{v\in V(G)}|T([G,o,v])|\mu(d[G,o])<\infty.
\ee
Hence, the mass-transport principle applies (see \cite{Aldous2007}, Section~2), and gives
\[
\int_{\mathcal G_\ast}\sum_{v\sim o} T([G,o,v])\mu(d[G,o])
=
\int_{\mathcal G_\ast}\sum_{v\sim o} T([G,v,o])\mu(d[G,o]).
\]
Thus, by \eqref{eq:T}, the last two terms in \eqref{eq:proof-first-var-2} are equal, and therefore \eqref{eq:proof-first-var-2}
reduces exactly to \eqref{eq:first-variation}.
\end{proof}

\section{Proof of Theorem~\ref{thm:main}}\label{sec:proof-main}

This section focuses on the proof of Theorem~\ref{thm:main}. We start by proving two auxiliary lemmas.
Recall that $F$ is the representative-player first-order operator, defined in \eqref{eq:F}. 
\begin{lemma}\label{lemma:VI-minimizer}
Suppose that Assumptions \ref{ass:degree-bound} and \ref{ass:potential} hold, and let \(x^\ast\in\mathcal K\) be a minimizer of \(\Phi\) over \(\mathcal K\). Then
\begin{equation} \label{eq:VI-H}
\langle F(x^\ast),z-x^\ast\rangle_{\mathcal H}\ge 0,
\quad \text{for all } z\in\mathcal K.
\end{equation}
\end{lemma}

\begin{proof}
Let \(z\in\mathcal K\) and define the perturbation $y:=z-x^\ast\in\mathcal H$.
Because \(x^\ast\) minimizes \(\Phi\) over the convex set \(\mathcal K\), the function \(t\mapsto \Phi(x^\ast+ty)\) attains its minimum at \(t=0\) on \([0,1]\). Hence its right derivative at \(0\) is nonnegative, that is
\[
\langle \nabla\Phi(x^\ast),y\rangle_{\mathcal H}
=
\lim_{t\downarrow 0}\frac{\Phi(x^\ast+ty)-\Phi(x^\ast)}{t}
\ge 0.
\] 
By Proposition \ref{prop:Phi-Gateaux},
\[
\langle \nabla\Phi(x^\ast),y\rangle_{\mathcal H}
=
\langle F(x^\ast),y\rangle_{\mathcal H}
=
\langle F(x^\ast),z-x^\ast\rangle_{\mathcal H}.
\]
This proves \eqref{eq:VI-H}.
\end{proof}

\begin{lemma}\label{lemma:pointwise-VI}
Suppose that Assumptions \ref{ass:degree-bound} and \ref{ass:potential} hold, and let \(x^\ast\in\mathcal K\). If
\be\label{eq:prop-vi}
\langle F(x^\ast),z-x^\ast\rangle_{\mathcal H}\ge 0,
\quad \text{for all } z\in\mathcal K,
\ee
then for \(\mu\)-a.e.\ \([G,o]\),
\begin{equation} \label{eq:pointwise-root-VI}
\left\langle
F(x^\ast)([G,o]),\, a-x^\ast([G,o])
\right\rangle_H
\ge 0,
\quad \text{for all } a\in K.
\end{equation}
\end{lemma}

\begin{proof}
Let
\[
A:=
\left\{
[G,o]\in\mathcal G_\ast :
\exists a\in K \text{ such that }
\left\langle
F(x^\ast)([G,o]),\, a-x^\ast([G,o])
\right\rangle_H<0
\right\}.
\]
We show that \(\mu(A)=0\). Suppose by contradiction that \(\mu(A)>0\). Since \(H\) is separable and \(K\subset H\) is closed, there exists a countable dense subset \((a_n)_{n\ge 1}\subset K\). Hence
\be\label{eq:An}
A=\bigcup_{n=1}^\infty A_n,
\quad
A_n:=
\left\{
[G,o]:
\left\langle
F(x^\ast)([G,o]),\, a_n-x^\ast([G,o])
\right\rangle_H<0
\right\}.
\ee
Therefore \(\mu(A_n)>0\) for some \(n\). Fix such an \(n\), and define
\[
z([G,o]):=
\begin{cases}
a_n, & [G,o]\in A_n,\\
x^\ast([G,o]), & [G,o]\notin A_n.
\end{cases}
\]
Since \(a_n\in K\) and \(x^\ast([G,o])\in K\) for \(\mu\)-a.e.~$[G,o]$, we have \(z\in\mathcal K\). Applying the variational inequality \eqref{eq:prop-vi} to this choice of \(z\), we obtain
\[
0\le
\langle F(x^\ast),z-x^\ast\rangle_{\mathcal H}
=
\int_{A_n}
\left\langle
F(x^\ast)([G,o]),\, a_n-x^\ast([G,o])
\right\rangle_H
\mu(d[G,o]),
\]
but the integrand is strictly negative on \(A_n\) by \eqref{eq:An}, which is a contradiction. Hence \(\mu(A)=0\), proving \eqref{eq:pointwise-root-VI}.
\end{proof}

We are now ready to prove Theorem~\ref{thm:main}.

\begin{proof}[Proof of Theorem~\ref{thm:main}]
By Lemma \ref{lemma:VI-minimizer}, the implication (i)$\implies$(ii) holds. 

Next, assume (ii), that is,
\[
\langle F(x^\ast),z-x^\ast\rangle_{\mathcal H}\ge 0,
\quad \text{for all } z\in\mathcal K.
\]
Hence Lemma \ref{lemma:pointwise-VI} yields
\be\label{eq:VI-F}
\left\langle
F(x^\ast)([G,o]),\, a-x^\ast([G,o])
\right\rangle_H
\ge 0,
\quad \text{for all } a\in K, \text{ for } \mu\text{-a.e. } [G,o],
\ee
Fix a rooted graph \([G,o]\) in the full-measure set where \eqref{eq:VI-F} holds, and define
\be\label{eq:Gamma}
\Gamma_{[G,o]}(a)
:=
f([G,o],a)+\sum_{v\sim o} h([G,o,v],a,x^\ast([G,v])),
\quad a\in H.
\ee
By Assumption \ref{ass:convex}, the map \(a\mapsto \Gamma_{[G,o]}(a)\) is convex, and by \eqref{eq:F},
\[
\nabla \Gamma_{[G,o]}(x^\ast([G,o]))=F(x^\ast)([G,o]).
\]
Therefore, by \eqref{eq:VI-F},
\be\label{eq:VI-Gamma}
\left\langle
\nabla \Gamma_{[G,o]}(x^\ast([G,o])),\, a-x^\ast([G,o])
\right\rangle_H\ge 0,
\quad \text{for all } a\in K.
\ee
It follows from \eqref{eq:VI-Gamma} that \(x^\ast([G,o])\) minimizes \(\Gamma_{[G,o]}\) over \(K\) (see \cite{bauschke2017}, Chapter~17, Proposition~17.51, Chapter~27, Proposition~27.8). Equivalently, recalling \eqref{eq:J} and \eqref{eq:Gamma} we have, 
\[
J_o^G(x^{\ast,G})\le J_o^G(a,x^{\ast,G}_{-o}),
\quad \text{for all } a\in K.
\]
Hence \(x^\ast\) is a quenched Nash equilibrium, so (iii) holds.

Conversely, assume that (iii) holds, i.e., let \(x^\ast\in\mathcal K\) be a quenched Nash equilibrium. Then for \(\mu\)-a.e.~\([G,o]\),
\[
J_o^G(x^{\ast,G})\le J_o^G(a,x^{\ast,G}_{-o}),
\quad \text{for all } a\in K.
\]
By \eqref{eq:J} and \eqref{eq:Gamma},
this means that \(x^\ast([G,o])\) minimizes \(\Gamma_{[G,o]}\) over \(K\). Since \(\Gamma_{[G,o]}\) is G\^ateaux differentiable and convex by Assumptions \ref{ass:potential} and \ref{ass:convex}, we get that
\[
\left\langle
\nabla \Gamma_{[G,o]}(x^\ast([G,o])),\, a-x^\ast([G,o])
\right\rangle_H\ge 0,
\quad \text{for all } a\in K, \text{ for } \mu\text{-a.e. } [G,o]
\]
(see \cite{bauschke2017}, Chapter~27, Proposition~27.8). Equivalently, by \eqref{eq:F},
\be\label{eq:VI-intermed}
\left\langle
F(x^\ast)([G,o]),\, a-x^\ast([G,o])
\right\rangle_H\ge 0,
\quad \text{for all } a\in K, \text{ for } \mu\text{-a.e. } [G,o].
\ee
Now let \(z\in\mathcal K\). Since \(z([G,o])\in K\) for \(\mu\)-a.e.~\([G,o]\), we may substitute \(a=z([G,o])\) in \eqref{eq:VI-intermed} and integrate with respect to $\mu$ to obtain
\[
\langle F(x^\ast),z-x^\ast\rangle_{\mathcal H}\ge 0,
\quad \text{for all } z\in\mathcal K,
\]
which is exactly (ii). 

Now assume that (ii) holds and, additionally, that \(\Phi\) is convex on \(\mathcal K\). Then, by Proposition \ref{prop:Phi-Gateaux},
\[
\langle \nabla\Phi(x^\ast),z-x^\ast\rangle_{\mathcal H}
=
\langle F(x^\ast),z-x^\ast\rangle_{\mathcal H}\ge 0,
\quad \text{for all } z\in\mathcal K.
\]
Since \(\Phi\) is convex, it follows that \(x^\ast\) minimizes \(\Phi\) over \(\mathcal K\) (see \cite{bauschke2017}, Chapter~27, Proposition~27.8), so (i) holds.
\end{proof}

\section{Proof of Theorem~\ref{thm:thermo-limit}}\label{sec:proof-conv}

Next, we prove Theorem~\ref{thm:thermo-limit}, establishing the thermodynamic limit.

\begin{proof}[Proof of Theorem~\ref{thm:thermo-limit}]
For \(x:\mathcal G_\ast\to H\), define $\Xi_x:\mathcal G_\ast\to\R$ by
\begin{equation}\label{eq:Xi-thermo}
\Xi_x([G,o])
:=
f\big([G,o],x([G,o])\big)
+
\frac12\sum_{v\sim o}\psi\big([G,o,v],x([G,o]),x([G,v])\big).
\end{equation}
Then
\be\label{eq:potential-xi}
\Phi_n(x)=\int_{\mathcal G_\ast}\Xi_x([G,o])\mu_n(d[G,o]),
\quad
\Phi(x)=\int_{\mathcal G_\ast}\Xi_x([G,o])\mu(d[G,o]).
\ee

\textbf{Step 1.} We first assume that $x$ is continuous. In this case, \(\Xi_x\) in \eqref{eq:Xi-thermo} is continuous on $\mathcal{G}_\ast$, since $f$, $\psi$, and $x$ are continuous. Next, arguing exactly as in the proof of Proposition~\ref{prop:Phi-finite}, by \eqref{eq:f-bound}, \eqref{eq:psi-bound}, and the Cauchy-Schwarz inequality, there is a constant \(C<\infty\) such that
\begin{equation}\begin{aligned}\label{eq:thermo-bound}
|\Xi_x([G,o])|
\le
C\Big(&
1+|f([G,o],a_\ast)|+\|\nabla f([G,o],a_\ast)\|_H^2
+\sum_{v\sim o}|\psi([G,o,v],a_\ast,a_\ast)|\\
&+\sum_{v\sim o}\|\nabla_1 h([G,o,v],a_\ast,a_\ast)\|_H^2+\|x_o\|_H^2+\sum_{v\sim o}\|x_v\|_H^2
\Big),
\end{aligned}\end{equation}
where \(x_o:=x([G,o])\) and \(x_v:=x([G,v])\). Set \(\delta:=\eps/2\). Since \(2(1+\delta)=2+\eps\), using the elementary inequality \((a_1+\dots+a_\ell)^{1+\delta}\le C_\delta\sum_{i=1}^\ell a_i^{1+\delta}\) for nonnegative \(a_i\), it follows from \eqref{eq:thermo-base-f}, \eqref{eq:thermo-base-psi}, \eqref{eq:thermo-base-h}, the first part of \eqref{eq:thermo-x}, and \eqref{eq:thermo-bound} that
\be\label{eq:Xix-bound}
\sup_{n\in\N}\int_{\mathcal G_\ast}|\Xi_x([G,o])|^{1+\delta}\mu_n(d[G,o])<\infty.
\ee
Here, the supremum of the last term on the right-hand side of \eqref{eq:thermo-bound} is finite by the mass-transport principle \eqref{eq:MTP} and the first part of \eqref{eq:thermo-x} as follows,
\be\begin{aligned}
\sup_{n\in\N}\int_{\mathcal G_\ast}\sum_{v\sim o}\|x_v\|_H^{2+\eps}\mu_n(d[G,o])
&=
\sup_{n\in\N}\int_{\mathcal G_\ast}\sum_{v\sim o}\|x_o\|_H^{2+\eps}\mu_n(d[G,o])\\
&\le
\sup_{n\in\N}D\int_{\mathcal G_\ast}\|x_o\|_H^{2+\eps}\mu_n(d[G,o])<\infty.
\end{aligned}\ee
For \(M>0\), define the truncation $\Xi_x^M:\mathcal G_\ast\to [-M,M]$ by
\be\label{eq:XiM-thermo}
\Xi_x^M([G,o]):=(-M)\vee \Xi_x([G,o])\wedge M.
\ee
Since \(\Xi_x\) is continuous, \(\Xi_x^M\) is bounded and continuous. Therefore, by the weak convergence of \(\mu_n\) to \(\mu\),
\begin{equation}\label{eq:trunc-conv}
\int_{\mathcal G_\ast}\Xi_x^M([G,o])\mu_n(d[G,o])
\longrightarrow
\int_{\mathcal G_\ast}\Xi_x^M([G,o])\mu(d[G,o]),
\quad \text{as } n \to \infty.
\end{equation}
Moreover, since \(|\Xi_x|^{1+\delta}\wedge M\) is also bounded and continuous, we have by \eqref{eq:Xix-bound},
\be\begin{aligned}\label{eq:trunc-bound}
\int_{\mathcal G_\ast}\big(|\Xi_x([G,o])|^{1+\delta}\wedge M\big)\mu(d[G,o])
&=
\lim_{n\to\infty}
\int_{\mathcal G_\ast}\big(|\Xi_x([G,o])|^{1+\delta}\wedge M\big)\mu_n(d[G,o])\\
&\le
\sup_{n\in\N}\int_{\mathcal G_\ast}|\Xi_x([G,o])|^{1+\delta}\mu_n(d[G,o])<\infty.
\end{aligned}\ee
Letting \(M\to\infty\) in \eqref{eq:trunc-bound}, the monotone convergence theorem yields \(\Xi_x\in L^{1+\delta}(\mu,\R)\), and hence \(\Xi_x\in L^1(\mu,\R)\). Moreover, since
\[
|\Xi_x([G,o])-\Xi_x^M([G,o])|
\le |\Xi_x([G,o])|\,\mathbf 1_{\{|\Xi_x([G,o])|>M\}},
\]
we obtain, for every \(n\in\N\), by Hölder's inequality,
\be\begin{aligned}\label{eq:Holder}
&\int_{\mathcal G_\ast}|\Xi_x([G,o])-\Xi_x^M([G,o])|\mu_n(d[G,o])\\
&\leq
\int_{\mathcal G_\ast}|\Xi_x([G,o])|\,\mathbf 1_{\{|\Xi_x([G,o])|>M\}}\mu_n(d[G,o])\\
&\leq
\Big(\int_{\mathcal G_\ast}|\Xi_x([G,o])|^{1+\delta}\mu_n(d[G,o])\Big)^{\frac1{1+\delta}}
\mu_n\big(|\Xi_x([G,o])|>M\big)^{\frac{\delta}{1+\delta}}.
\end{aligned}\ee
Using Markov's inequality,
\be\label{eq:Markov}
\mu_n\big(|\Xi_x([G,o])|>M\big)
\leq
\frac{1}{M^{1+\delta}}
\int_{\mathcal G_\ast}|\Xi_x([G,o])|^{1+\delta}\mu_n(d[G,o]),
\ee
and therefore by \eqref{eq:Holder} and \eqref{eq:Markov},
\be\label{eq:diff}
\int_{\mathcal G_\ast}|\Xi_x([G,o])-\Xi_x^M([G,o])|\mu_n(d[G,o])
\leq
\frac{1}{M^\delta}
\int_{\mathcal G_\ast}|\Xi_x([G,o])|^{1+\delta}\mu_n(d[G,o]).
\ee
Hence, by \eqref{eq:Xix-bound} and \eqref{eq:diff},
\begin{equation}\begin{aligned}\label{eq:tail-unif}
\sup_{n\ge1}\int_{\mathcal G_\ast}|\Xi_x([G,o])-\Xi_x^M([G,o])|\mu_n(d[G,o])\le \frac{1}{M^\delta}\sup_{n\in\N}\int_{\mathcal G_\ast}|\Xi_x([G,o])|^{1+\delta}\mu_n(d[G,o])\to 0,
\end{aligned}\end{equation}
as $M\to\infty$. Similarly, since \(\Xi_x\in L^1(\mu,\R)\), the dominated convergence theorem implies
\begin{equation}\label{eq:tail-mu}
\lim_{M\to\infty}\int_{\mathcal G_\ast}|\Xi_x([G,o])-\Xi_x^M([G,o])|\mu(d[G,o])=0.
\end{equation}
Finally, by \eqref{eq:potential-xi}, omitting the dependence on $[G,o]$ in the notation, for every \(M>0\),
\be\begin{aligned}\label{eq:tail-all}
|\Phi_n(x)-\Phi(x)|
\le
\int_{\mathcal G_\ast}|\Xi_x-\Xi_x^M|d\mu_n
+
\left|
\int_{\mathcal G_\ast}\Xi_x^Md\mu_n
-
\int_{\mathcal G_\ast}\Xi_x^Md\mu
\right|
+
\int_{\mathcal G_\ast}|\Xi_x-\Xi_x^M|d\mu.
\end{aligned}\ee
Letting first \(n\to\infty\), we obtain from \eqref{eq:trunc-conv} and \eqref{eq:tail-all},
\be\label{eq:limsup}
\limsup_{n\to\infty}|\Phi_n(x)-\Phi(x)|
\le
\sup_{n\in\N}\int_{\mathcal G_\ast}|\Xi_x-\Xi_x^M|d\mu_n
+
\int_{\mathcal G_\ast}|\Xi_x-\Xi_x^M|d\mu.
\ee
Now let \(M\to\infty\). By \eqref{eq:tail-unif} and \eqref{eq:tail-mu}, the right-hand side of \eqref{eq:limsup} converges to \(0\). This completes the proof for continuous $x$.

\textbf{Step 2.} We now consider the general case where $x$ is only assumed to be measurable. Recall the definition of \(\Xi_x\) in \eqref{eq:Xi-thermo}. Arguing exactly as in Step 1, there exists \(C<\infty\) such that \eqref{eq:thermo-bound} holds. Setting \(\delta:=\eps/2\), it follows from \eqref{eq:thermo-base-f}, \eqref{eq:thermo-base-psi}, \eqref{eq:thermo-base-h}, \eqref{eq:thermo-x}, \eqref{eq:thermo-bound}, and the mass-transport principle \eqref{eq:MTP} that
\begin{equation}\label{eq:UI-meas}
\sup_{n\in\N}\int_{\mathcal G_\ast}|\Xi_x([G,o])|^{1+\delta}\mu_n(d[G,o])<\infty
\quad\text{and}\quad
\Xi_x\in L^{1+\delta}(\mu,\R).
\end{equation}
Note that we have used both parts of \eqref{eq:thermo-x} in order to derive \eqref{eq:UI-meas}, while only the first part of \eqref{eq:thermo-x} was needed in Step 1. Indeed, in Step 1 the statement $\Xi_x\in L^{1+\delta}(\mu,\R)$ could be inferred from \eqref{eq:trunc-bound} and the monotone convergence theorem. 

Recall the definition of \(\Xi_x^M\) in \eqref{eq:XiM-thermo}.
By \eqref{eq:UI-meas}, it follows as in \eqref{eq:tail-unif} and \eqref{eq:tail-mu} that
\begin{equation}\label{eq:UI-tail-meas}
\lim_{M\to\infty}\sup_{n\in\N}\int_{\mathcal G_\ast}|\Xi_x-\Xi_x^M|d\mu_n=0
\quad\text{and}\quad
\lim_{M\to\infty}\int_{\mathcal G_\ast}|\Xi_x-\Xi_x^M|d\mu=0.
\end{equation}
Set
\[
\mathcal G_\ast^D:=\Big\{[G,o]\in\mathcal G_\ast:\deg_G(v)\le D\text{ for all }v\in V(G)\Big\}.
\]
By Assumptions~\ref{ass:degree-bound} and \ref{ass:thermo}, the measures \(\mu\) and \(\mu_n\) are supported on \(\mathcal G_\ast^D\). Moreover, \(\mathcal G_\ast^D\) is compact under the metric \(d_\ast\) (see \cite{lovasz2012large}, Chapter~18.3.1). Fix \(\rho>0\). Since \(\mathcal G_\ast^D\) is compact, Lusin's theorem (see \cite{folland1999real}, Chapter 7, Theorem 7.10) yields a compact set \( L\subset\mathcal{G}_\ast^D\subset\mathcal G_\ast\) such that
\be\label{eq:rho/2-1}
\mu(L^c)<\frac{\rho}{2},
\ee
and the restriction $x|_L:L\to H$ is continuous. Therefore, by the Dugundji extension theorem (see \cite{dugundji1951extension}, Theorem~4.1), there exists a continuous map $x_\rho:\mathcal G_\ast\to H$ such that $x_\rho|_L=x|_L$. We next shrink \(L\) slightly in order to obtain a compact \(\mu\)-continuity set. For any \(t>0\), define
\be\label{eq:Lt}
L^t=\Big\{[G,o]\in L\,\Big|\, d_\ast([G,o],L^c)\ge t\Big\}.
\ee
Then, since the map $[G,o]\mapsto d_\ast([G,o],L^c)$ is continuous, for each $t$, the set  \(L^t\subset L\) is closed, and therefore compact. Moreover, \(L^t\uparrow L\) as \(t\downarrow 0\). Hence, by continuity from below of $\mu$, for sufficiently small \(t>0\),
\be\label{eq:rho/2-2}
\mu(L\setminus L^t)<\frac{\rho}{2}.
\ee
Next, observe that by \eqref{eq:Lt},
\be\label{eq:partialLt}
\partial L^t\subset\Big\{[G,o]\in L\,\Big|\, d_\ast([G,o],L^c)= t\Big\}.
\ee
The level sets on the right-hand side of \eqref{eq:partialLt} are pairwise disjoint for different values of \(t\). Since \(\mu\) is finite, it follows that there are at most countably many \(t\in(0,\infty)\) for which
\[
\mu\Big(\big\{[G,o]\in L\,\big|\, d_\ast([G,o],L^c)= t\big\}\Big)>0.
\]
Therefore, by \eqref{eq:partialLt}, we can choose \(t>0\) sufficiently small such that both \eqref{eq:rho/2-2} holds and
\[
\mu(\partial L^t)=0.
\]
For such $t$ denote $L_\rho:=L^t$. Then \(L_\rho\subset L\), so
\be\label{eq:Lusin-good}
x_\rho|_{L_\rho}=x|_{L_\rho}.
\ee
Moreover, \(L_\rho\) is compact, and by \eqref{eq:rho/2-1} and \eqref{eq:rho/2-2},
\be \label{gg1} 
\mu(L_\rho^c)= \mu(L^c)+\mu(L\setminus L_\rho)<\rho.
\ee
Finally, since \(\mu(\partial L_\rho)=0\), the set \(L_\rho\) is a \(\mu\)-continuity set. Hence, by the weak convergence \(\smash{\mu_n\xrightarrow{w}\mu}\) and the Portmanteau theorem,
\begin{equation}\label{eq:L-conv}
\mu_n(L_\rho^c)\longrightarrow \mu(L_\rho^c), \quad \text{as } n \to \infty.
\end{equation}
Now define
\be\label{eq:Arho}
A_\rho
:=
\Big\{
[G,o]\in\mathcal G_\ast\,\Big|\,[G,o]\notin L_\rho
\ \text{or }[G,v]\notin L_\rho \text{ for some } v\sim o 
\Big\}.
\ee
Next, fix \(M>0\). Then, by \eqref{eq:Lusin-good}, we have
\[
\Xi_x^M([G,o])=\Xi_{x_\rho}^M([G,o]),
\quad\text{for } [G,o]\in A_\rho^c.
\]
Therefore,
\begin{equation}\label{eq:trunc-diff-rho}
|\Xi_x^M([G,o])-\Xi_{x_\rho}^M([G,o])|
\le 2M\,\mathbf 1_{A_\rho}([G,o]).
\end{equation}
Moreover, using the mass-transport principle \eqref{eq:MTP} and Assumption~\ref{ass:thermo}, it follows from \eqref{eq:Arho},
\begin{align}
\mu_n(A_\rho)
&\le
\mu_n(L_\rho^c)
+
\int_{\mathcal G_\ast}\sum_{v\sim o}\mathbf 1_{L_\rho^c}([G,v])\mu_n(d[G,o]) \nonumber\\
&=
\mu_n(L_\rho^c)
+
\int_{\mathcal G_\ast}\sum_{v\sim o}\mathbf 1_{L_\rho^c}([G,o])\mu_n(d[G,o]) \nonumber\\
&\le
(D+1)\mu_n(L_\rho^c).
\label{eq:Arho-bound}
\end{align}
Hence, from \eqref{gg1} and \eqref{eq:L-conv} we get, 
\begin{equation}\label{eq:Arho-limsup}
\limsup_{n\to\infty}\mu_n(A_\rho)\le (D+1)\rho.
\end{equation}
Similarly,
\begin{equation}\label{eq:Arho-mu}
\mu(A_\rho)\le (D+1)\rho.
\end{equation}
Now, since \(x_\rho\) is continuous, \(\Xi_{x_\rho}^M\) is bounded and continuous on \(\mathcal G_\ast\). Therefore, by weak convergence,
\begin{equation}\label{eq:cont-approx-conv}
\int_{\mathcal G_\ast}\Xi_{x_\rho}^Md\mu_n
\longrightarrow
\int_{\mathcal G_\ast}\Xi_{x_\rho}^Md\mu.
\end{equation}
Finally, it holds for every \(n\in\N\),
\begin{align*}
\left|\int_{\mathcal G_\ast}\Xi_xd\mu_n-\int_{\mathcal G_\ast}\Xi_xd\mu\right|
&\le
\int_{\mathcal G_\ast}|\Xi_x-\Xi_x^M|d\mu_n
+
\int_{\mathcal G_\ast}|\Xi_x^M-\Xi_{x_\rho}^M|d\mu_n \\
&\quad
+
\left|
\int_{\mathcal G_\ast}\Xi_{x_\rho}^Md\mu_n
-
\int_{\mathcal G_\ast}\Xi_{x_\rho}^Md\mu
\right| \\
&\quad
+
\int_{\mathcal G_\ast}|\Xi_{x_\rho}^M-\Xi_x^M|d\mu
+
\int_{\mathcal G_\ast}|\Xi_x^M-\Xi_x|d\mu.
\end{align*}
Using \eqref{eq:trunc-diff-rho}, \eqref{eq:Arho-limsup}, \eqref{eq:Arho-mu}, and \eqref{eq:cont-approx-conv}, we obtain by letting \(n\to\infty\),
\be\label{eq:bound-final}
\limsup_{n\to\infty}
\left|\int_{\mathcal G_\ast}\Xi_xd\mu_n-\int_{\mathcal G_\ast}\Xi_xd\mu\right|
\le
\sup_{n\in\N}\int_{\mathcal G_\ast}|\Xi_x-\Xi_x^M|d\mu_n
+
4M(D+1)\rho
+
\int_{\mathcal G_\ast}|\Xi_x^M-\Xi_x|d\mu.
\ee
Now, in \eqref{eq:bound-final}, let first \(\rho\to 0\), and then \(M\to\infty\). By \eqref{eq:UI-tail-meas}, the right-hand side of \eqref{eq:bound-final} then converges to \(0\). Hence
\[
\int_{\mathcal G_\ast}\Xi_x([G,o])\mu_n(d[G,o])
\longrightarrow
\int_{\mathcal G_\ast}\Xi_x([G,o])\mu(d[G,o]),\quad \text{as } n \to \infty.
\]
Recalling \eqref{eq:potential-xi}, this is exactly \(\Phi_n(x)\to\Phi(x)\), as $n\rightarrow \infty$.
\end{proof}

\section{Proofs of Section~\ref{sec:examples}}\label{sec:proofs-examples}

We now present the proofs of the results in Section~\ref{sec:examples}. We first prove Proposition~\ref{prop:lq-example}.
\begin{proof}[Proof of Proposition~\ref{prop:lq-example}]
Recalling \eqref{eq:lq-f}, we have
\be\label{eq:lq-nabla-f}
\nabla f([G,o],a)=-\eta([G,o]),
\ee
so \(a\mapsto f([G,o],a)\) is G\^ateaux differentiable on \(\mathbb R\), and \eqref{eq:lip-f} holds with \(L_f=0\). Next, from \eqref{eq:lq-h},
\[
\nabla_1 h([G,o,v],a,b)=a-b,
\quad
\nabla_2 h([G,o,v],a,b)=b-a.
\]
Since \(\psi=h\), it follows that
\be\label{eq:lq-nabla-psi}
\nabla_1\psi([G,o,v],a,b)=a-b=\nabla_1 h([G,o,v],a,b),
\ee
and
\[
\nabla_2\psi([G,o,v],a,b)=b-a=\nabla_1 h([G,v,o],b,a),
\]
so \eqref{eq:potential-id-1}--\eqref{eq:potential-id-2} hold. Moreover,
\[
|\nabla_1 h([G,o,v],a,b)-\nabla_1 h([G,o,v],a',b')|
=
|(a-b)-(a'-b')|
\le |a-a'|+|b-b'|,
\]
so \eqref{eq:lip-h} holds with \(L_h=1\). Assumption \ref{ass:convex} also holds, since for every family \((b_v)_{v\sim o}\subset\mathbb R\), the map
\[
a\mapsto -\eta([G,o])\,a+\frac12\sum_{v\sim o}(a-b_v)^2
\]
is convex. Now \eqref{eq:lq-potential} follows directly from \eqref{eq:Phi}, because \(\psi=h\) and
\[
\frac12\sum_{v\sim o}\psi([G,o,v],x([G,o]),x([G,v]))
=
\frac14\sum_{v\sim o}\big(x([G,o])-x([G,v])\big)^2.
\]
Moreover, it follows from \eqref{eq:F}, \eqref{eq:lq-nabla-f}, and \eqref{eq:lq-nabla-psi} that the representative-player first-order operator \(F\) takes the form
\begin{equation}\label{eq:lq-F}
F(x)([G,o])
=
-\eta([G,o])+\sum_{v\sim o}\big(x([G,o])-x([G,v])\big).
\end{equation}
Thus, if \(x^\ast\in\mathcal{H}\) satisfies 
\begin{equation}\label{eq:lq-EL}
\sum_{v\sim o}\big(x^\ast([G,o])-x^\ast([G,v])\big)=\eta([G,o]),
\quad \text{for } \mu\text{-a.e. } [G,o],
\end{equation}
then by \eqref{eq:lq-F} we have \(F(x^\ast)=0\). Furthermore, \(\Phi\) is convex by \eqref{eq:lq-potential}. Hence \(x^\ast\) minimizes the potential functional \eqref{eq:lq-potential} (see \cite{bauschke2017}, Chapter~27, Proposition~27.8). By Theorem \ref{thm:main}, it follows that \(x^\ast\) is a quenched Nash equilibrium of the game. Finally, dividing \eqref{eq:lq-EL} by \(\deg_G(o)\) yields
\[
x^\ast([G,o])-\frac1{\deg_G(o)}\sum_{v\sim o}x^\ast([G,v])
=
\frac{\eta([G,o])}{\deg_G(o)},
\]
which, recalling \eqref{eq:lq-P} and \eqref{eq:lq-Delta}, is exactly \eqref{eq:lq-poisson}.
\end{proof}

We continue with the proof of Corollary~\ref{cor:green}.

\begin{proof}[Proof of Corollary~\ref{cor:green}]
Recalling \eqref{eq:lq-spectral-radius}, first fix a graph \(G\) in the full-measure set where \(\rho(P_G)<1\). Since the simple random walk on \(G\) is reversible with respect to the measure \(m_G\) from \eqref{eq:lq-mG}, it follows that the Green kernel \eqref{eq:lq-green-kernel} defines a bounded linear operator on \(\ell^2(V(G),m_G)\), given by
\begin{equation}\label{eq:lq-GG}
(K_G \xi)(v):=\sum_{u\in V(G)}K_G(v,u)\xi(u),\quad v\in V(G)
\end{equation}
(see \cite{woess2000random}, Chapter II, Section 10.A, Theorem~10.3). Recalling \eqref{eq:lq-DeltaG}, it holds that 
\be\label{eq:lq-GG-eq}
K_G=(I_G-P_G)^{-1}=(\Delta_G)^{-1}
\ee
is bounded linear operator on \(\ell^2(V(G),m_G)\) (see \cite{woess2000random}, Chapter I, Section 2.A, Equation 2.7). Now define $\xi^G$ as in \eqref{eq:xi}. Hence, if \(x^{\ast,G}\) is given by \eqref{eq:lq-green}, then
$x^{\ast,G}=K_G \xi^G$ by \eqref{eq:lq-GG}, and therefore by \eqref{eq:lq-GG-eq},
\[
(I_G-P_G)x^{\ast,G}=\xi^G.
\]
Equivalently,
\be\label{eq:lq-vpoi}
(\Delta_G x^{\ast,G})(v)=\xi^G_v=\frac{\eta([G,v])}{\deg_G(v)},
\quad v\in V(G).
\ee
Evaluating \eqref{eq:lq-vpoi} at the root and recalling \eqref{eq:xGv}, we obtain
\[
(\Delta x^\ast)([G,o])=\frac{\eta([G,o])}{\deg_G(o)},
\quad \text{for } \mu\text{-a.e. } [G,o].
\]
Thus \eqref{eq:lq-poisson} holds, and Proposition~\ref{prop:lq-example} yields that \(x^\ast\) is a quenched Nash equilibrium.
\end{proof}

We conclude with the proof of Proposition~\ref{prop:W-example}.

\begin{proof}[Proof of Proposition~\ref{prop:W-example}]
Let \(x^\ast\in\mathcal H\) satisfy \eqref{eq:W-poisson}. Recalling \eqref{eq:W-Delta}, fix a rooted graph \([G,o]\) in the full-measure set where \eqref{eq:W-poisson} holds, and define $\Gamma_{[G,o]}:\R\to\R$ by
\be\begin{aligned}\label{eq:W-Gamma}
\Gamma_{[G,o]}(a)
:&=
f([G,o],a)+\sum_{v\sim o}h([G,o,v],a,x^\ast([G,v]))
\\&=
-\eta([G,o])\,a+\sum_{v\sim o}W\!\big(a-x^\ast([G,v])\big),
\end{aligned}\ee
where we used \eqref{eq:W-f} and \eqref{eq:W-h}. Since \(W\) is convex, the map \(a\mapsto \Gamma_{[G,o]}(a)\) is convex. Moreover, by \eqref{eq:W-Gamma},
\[
\Gamma_{[G,o]}'(a)
=
-\eta([G,o])+\sum_{v\sim o}W'\!\big(a-x^\ast([G,v])\big).
\]
Using \eqref{eq:W-Delta} and \eqref{eq:W-poisson}, we obtain
\[
\Gamma_{[G,o]}'\big(x^\ast([G,o])\big)
=
-\eta([G,o])+\sum_{v\sim o}W'\!\big(x^\ast([G,o])-x^\ast([G,v])\big)
=
0.
\]
Therefore \(x^\ast([G,o])\) minimizes \(\Gamma_{[G,o]}\) over \(\mathbb R\). Moreover, by \eqref{eq:J} and \eqref{eq:W-Gamma}, for every \(a\in\mathbb R\),
\[
J_o^G(a,x_{-o}^{\ast,G})
=
f([G,o],a)+\sum_{v\sim o}h([G,o,v],a,x^\ast([G,v]))
=
\Gamma_{[G,o]}(a).
\]
Hence
\[
J_o^G(x^{\ast,G})=\Gamma_{[G,o]}(x^\ast([G,o]))
\le \Gamma_{[G,o]}(a)=J_o^G(a,x_{-o}^{\ast,G}),
\quad \text{for all } a\in\mathbb R.
\]
Since this holds for $\mu$-a.e.~\([G,o]\), \(x^\ast\) is a quenched Nash equilibrium.
\end{proof}

\end{document}